\documentclass{article}
\usepackage{hyperref,amsmath,amsthm,amsfonts,amscd,flafter,epsf,amssymb,wasysym,graphicx}
\theoremstyle{plain}
\newtheorem*{maintheorem1*}{Theorem A [Construction]} 
\newtheorem*{maintheorem2*}{Theorem B [Classification]}
\newtheorem*{thm*}{Theorem}
\newtheorem*{thma*}{Theorem A}
\newtheorem*{thmaa*}{Theorem A'}
\newtheorem*{thmb*}{Theorem B}
\newtheorem*{thmo*}{Theorem 1.1}
\newtheorem*{thmc*}{Theorem C}
\newtheorem*{thmd*}{Theorem D}
\newtheorem*{thmf*}{Theorem 4.1}
\newtheorem*{conjecture*}{Conjecture}
\newtheorem*{prop*}{Proposition}
\newtheorem{thm}{Theorem}
\newtheorem{cor}[thm]{Corollary}
\newtheorem{lem}[thm]{Lemma}
\newtheorem{prop}[thm]{Proposition}

\theoremstyle{definition}

\newtheorem*{proofc*}{Proof of Theorem C}

\newtheorem{question}[thm]{Question}

\newtheorem{definition}[thm]{Definition}

\newtheorem{remark}[thm]{Remark}

\newcommand{\leftexp}[2]{{\vphantom{#2}}^{#1}{#2}}
\def\o{\mathcal{O}}
\def\gcal{\mathcal{G}}

\def\pcal{\mathcal{P}}
\def\lcal{\mathcal{L}}
\def\cal{\mathcal{H}}
\def\bbz{\mathbb{Z}}
\def\bbq{\mathbb{Q}}

\def\bbr{\mathbb{R}}
\def\bba{\mathbb{A}}

\def\bbh{\mathbb{H}}

\def\bbg{\mathbb{G}}
\def\adbbg{\overline{\mathbb{G}}}
\def\bbt{\mathbb{T}}

\def\bbb{\mathbb{B}}

\def\gfr{\mathfrak{g}}
\def\hfr{\mathfrak{h}}
\def\Hfr{\mathfrak{H}}
\def\mfr{\mathfrak{m}}

\def\pfrl{\mathfrak{P}}
\def\pfr{\mathfrak{p}}
\def\f{\mathfrak{f}}
\def\F{\mathfrak{F}}
\def\L{\mathfrak{L}}
\def\sfr{\mathfrak{s}}

\def\vare{\varepsilon}
\def\adg{\overline{\mathbb{G}}}
\def\adgcal{\overline{\mathcal{G}}}
\def\Aut{{\rm Aut}}

\def\vol{{\rm vol}}

\def\Ccal{\mathcal{C}}

\def\dcal{\mathcal{D}}

\def\A{{\rm A}}
\def\Ad{{\rm Ad}}

\def\GL{{\rm GL}}

\def\rank{{\rm rank}}
\def\Spec{{\rm Spec}}
\def\Div{{\rm Div}}

\def\Cl{{\rm Cl}}
\def\cl{{\rm cl}}

\def\h{\hspace{1mm}}

\def\xvec{\overrightarrow{x}}
\def\cdim{{\rm codim}}
\def\lmbb{\overline{\mathbb{M}}}
\def\lmcal{\overline{\mathcal{M}}}

\title{Counting lattices in simple Lie groups:\\ the positive characteristic case.}
\author{Alireza Salehi Golsefidy\footnote{A. S-G. was partially supported by the NSF grant DMS-0635607 and the NSF grant DMS-1001598.}}
\date{July 29, 2011}
\begin{document}
\maketitle
\begin{abstract}
\noindent
In this article we prove a conjecture of A. Lubotzky:
let  $G=\bbg_0(K)$, where $K$ is a local field of characteristic $p\ge 5$, $\bbg_0$ is a simply connected, absolutely almost simple $K$-group of $K$-rank at least 2. We give the rate of growth of
 \[
\rho_x(G):=|\{\Gamma\subseteq G| \text{$\Gamma$ a lattice in $G$},\h \vol(G/\Gamma)\le x\}/\sim|, 
\]
where $\Gamma_1\sim \Gamma_2$ if and only if there is an abstract automorphism $\theta$ of $G$ such that $\Gamma_2=\theta(\Gamma_1)$. We also study the rate of subgroup growth $s_x(\Gamma)$ of any lattice $\Gamma$ in $G$. As a result we show that these two functions have the same rate of growth which proves Lubotzky\rq{}s conjecture. 

Along the way, we also study the rate of growth of the number of equivalence classes of maximal lattices in $G$ with covolume at most $x$.
\end{abstract}
\section{Introduction and statement of results.}
Let $G$ be a semisimple Lie group with a fixed Haar measure $\mu$. Let $\rho_x(G)$ be the number of lattices (i.e. discrete subgroups of finite covolume) in $G$ of
covolume at most $x$, up to an automorphism of G, i.e. the same definition as the one given in the abstract for any semisimple Lie group. Much attention has been 
given in recent years to the question of determining the rate of growth of $\rho_x(G)$ where $G$ 
is a simple real Lie group. In this paper, we focus on the case when $G$ is a simple group over a 
local field of positive characteristic and determine the rate of growth of $\rho_x(G)$. 
For the most general cases, our result depends on several well-founded conjectures, which are 
known to hold in most cases. We should say that in the characteristic zero case, as the group of 
outer automorphisms of $G$ is finite, one can instead count conjugacy 
classes of lattices. However in the positive characteristic the group of outer automorphisms 
$\Aut(G)/{\rm Inn}(G)$ is an infinite compact group, which adds more complications to all the arguments in this work.
\begin{thm}\label{t:CountingLattices}
Let $G=\bbg_0(K)$, where $K$ is a local field of characteristic $p\ge 5$ and $\bbg_0$ is a simply connected absolutely almost simple $K$-group of $K$-rank bigger than 1. Assuming the CSP, MP and the Weil conjecture hold for any group with the same absolute type as $\bbg_0$, there exist positive numbers $C$ and $D$ depending only on $G$ such that
\[
x^{C\log x}\le \rho_x(G) \le x^{D\log x},
\]
for all sufficiently large $x$. Moreover the lower bound is unconditional. 
\end{thm} 
\noindent
We will explain below what the Weil conjecture, CSP and MP are and when they are known to hold. Let us first put this theorem in the perspective of the previous works as it represents a phenomenon which is different than the characteristic zero case. 
\\

\noindent
M.~Burger, T.~Gelander, A.~Lubotzky, and S.~Mozes~\cite{BGLM} studied this problem for 
$G={\rm PO}(n,1)$, for $n\ge 4$. Their approach was geometric and they counted the torsion 
free lattices. As a result, they proved that the rate of growth of $\rho^{\circ}_x(G)$ is the same as the rate of growth of  the number $s_x(\Gamma)$ of subgroups of index at most $x$ in a certain lattice $\Gamma$ in $G$. Lubotzky conjectured that this should be a general phenomenon, namely the rate of growth of $\rho_x(G)$ should be the same as the rate of subgroup growth of any lattice in $G$ (for further developments for the arithmetic subgroups in ${\rm PO}(2,1)$ or ${\rm PO}(3,1)$ see~\cite{BGLS}). Belolipetsky and Lubotzky~\cite{BL} showed that this conjecture is not true in the characteristic zero case. In contrast, Theorem~\ref{t:CountingLattices} proves Lubotzky's conjecture in the positive characteristic case.
\\

\noindent
The general strategy is to divide the problem into three parts. First one has to understand the 
asymptotic behavior of $m_x(G)$, the number of maximal lattices, up to an automorphism of $G$, 
with covolume at most $x$. Second the rate of growth of $s_x(\Gamma)$, where $\Gamma$ is a 
lattice in $G$, has to be described. Finally these results have to be carefully combined to get the 
rate of growth of $\rho_x(G)$. Our result for counting  maximal lattices is:
\begin{thm}~\label{t:maximal}
For a given $G$ as above. Assuming the Weil conjecture holds for any group with the same 
absolute type as $\bbg_0$, there are positive numbers $A$ and $B$ which depend only on $G$ 
such that,
\[
x^A\le m_x(G)\le x^{B\log\log x},
\]
for all sufficiently large $x$. The lower bound holds unconditionally.
\end{thm}
\noindent
This is the characteristic $p$ analogue of Belolipetsky's~\cite{Be} result for the characteristic zero 
case; but we should stress that his proof relies on various results which are either only suitable for 
number fields or only known in that case. In the positive characteristic case, we have to employ 
certain results from algebraic geometry and prove a local-global theorem for adjoint quasi-split 
groups. We also give a method to get an upper bound for the class number of a coherent family of 
parahorics which also works in the characteristic zero case. As a result, in the characteristic $p$ 
case, our upper bound $x^{B\log\log x}$ is better than Belolipetsky's upper bound $x^{(\log x)^{\vare}}$. 
\\

\noindent
Another challenge we face is with the subgroup growth. Here the results on counting congruence 
subgroups were known only for globally split lattices. One of the difficulties lies on the fact that 
structure of certain graded Lie algebras is complicated if the global form is not split. In order to 
overcome this difficulty, conceptually, we make a base change to an unramified extension in order 
to get a quasi-split group. Then using results of Prasad-Raghunathan~\cite{PR} and changing the 
grading on the Lie algebra, we get the rate of $c_x(\bbg(\o(\pfr_0)))$ congruence subgroup 
growth of $\bbg(\o(\pfr_0))$, where $\pfr_0$ is a place of a global function field $k$, $\o(\pfr_0)$ 
is the ring of $\pfr_0$-integers of $k$, $\bbg$ is an absolutely almost simple $k$-group with a 
fixed $k$-embedding into $\mathbb{GL}_n$ and $\bbg(\o(\pfr_0))=\bbg\cap \GL_n(\o(\pfr_0))$. 
This result extends the result of Ab\'{e}rt, Nikolov and Szegedy~\cite{ANS}, who showed it for 
$k$-split groups.
\begin{thm}\label{t:SubgroupGrowth}
In the above setting, there exist positive numbers $C$ and $D$ depending on 
$\Lambda_0=\bbg(\o(\pfr_0))$ such that
\[
x^{C \log x}\le c_x(\Lambda_0)\le x^{D \log x},
\]
for all sufficiently large $x$. 
\end{thm}
\noindent
Indeed, our method of proof refines the above theorem and gives an explicit dependence of the constants on a given lattice in $G$. 
\begin{cor}\label{c:SubgroupGrowth}
Assuming the CSP, MP and the Weil conjecture hold for any group with the same absolute type as $\bbg_0$, there are positive numbers $C$, $D$, and $x_0$  depending only on $G$ such that for any lattice $\Gamma$ in $G$
\[
x^{C \log x/\log (\vol(G/\Gamma))}\le s_x({\Gamma})\le (\vol(G/\Gamma)\cdot x)^{D \log (\vol(G/\Gamma)\cdot x)},
\]
for any positive number $x\ge x_0$.
\end{cor}
\noindent
 After proving Theorem~\ref{t:CountingLattices} or Corollary~\ref{c:SubgroupGrowth}, the next natural questions are about their asymptotic behavior.
\begin{question}\label{q:CountingLattices}
In the above setting, does the following limit exist?
\[
\lim_{x\rightarrow \infty} \frac{\log \rho_x(G)}{(\log x)^2}
\]
\end{question}
\noindent
In the characteristic zero case, as a result of works of D.~Goldfeld, A. Lubotzky, N. Nikolov, L. Pyber~\cite{GLP,LN}, we know the asymptotic behavior of $s_x(\Gamma)$. It is interesting to understand the exact asymptotic behavior of $s_x(\Gamma)$ in the positive characteristic.
\begin{question}\label{q:SubgroupGrowth}
In the above setting, for a given $\Gamma$ a lattice in $G$, does the following limit exist?
\[
\lim_{x\rightarrow \infty} \frac{\log s_x(\Gamma)}{(\log x)^2}
\]
\end{question}
\noindent
One can also ask what the rate of growth of $m_x(G)$ is exactly. 
\begin{question}\label{q:Maximal}
In the above setting, are there positive numbers $A$ and $B$ depending on $G$ such that 
\[
x^A\le m_x(G) \le x^B,
\]
for sufficiently large $x$?
\end{question}
\noindent
In fact our proof shows that Question~\ref{q:Maximal} has an affirmative answer if a question of de Jong and Katz~\cite{dJK} and the Weil conjecture hold. de Jong and Katz~\cite{dJK} asked if for any $q$ there is $c=c(q)$ such that the number of smooth projective curves of genus $g$ over $\f_q$ is at most $c^g$.
\\

\noindent
In order to prove Theorems~\ref{t:CountingLattices} and \ref{t:SubgroupGrowth}, we will prove the following theorem on graded Lie algebras. It essentially says as we ``unwind" a $\bbz/m\bbz$-graded perfect finite dimensional Lie algebra, it does not lose its perfectness by much.
\begin{thm}\label{t:GradedLie}
Let $m$ be a natural number and $\widehat{\gfr}=\oplus^{m-1}_{i=0} \gfr_i$ be a perfect $\bbz/m\bbz$-graded $\mathfrak{F}$-Lie algebra, 
\[\mathfrak{L}=\oplus_{i=1}^{\infty} \gfr_i\otimes t^i,\]
 where $\gfr_i=\gfr_{i({\rm mod}\h m)}$, $D$ a positive integer, and $\hfr$ an $\mathfrak{F}$-sub-algebra with finite co-dimension in $\mathfrak{L}^D$, where $\mathfrak{L}^D$ denotes the direct sum of $D$ copies of $\mathfrak{L}$. There exists a constant $C=C(\widehat{\gfr})$ depending only on $\widehat{\gfr}$, such that 
 \[
 \cdim_{\mathfrak{L}^D} [\hfr,\hfr]\le C(\cdim_{\mathfrak{L}^D} \hfr+D).
 \]
\end{thm}
\noindent
Let us finish the introduction, by saying a few words on the mentioned conjectures. Let $k$ be global field and $\bbg$ a simply connected absolutely almost simple $k$-group. Weil conjectured that $\tau_k(\bbg)=1$, where $\tau_k(\bbg)$ is the Tamagawa number. We refer the reader to~\cite{We} for the exact definition of the Tamagwa number. Weil's conjecture is known to hold  when $\bbg/k$ is either an inner form of type A, any form of type B, of type C, of type D except triality forms of type ${\rm D}_4$, of type ${\rm G}_2$~\cite{We}, or a $k$-split group~\cite{Ha74}. Indeed what we need in this article is only a uniform lower bound for $\tau_k(\bbg)$.
\\

\noindent
Congruence subgroup property (CSP) essentially says that any arithmetic lattice has a subgroup of finite index such that any of its finite index subgroups is congruence and Margulis-Platonov (MP) conjecture describes structure of normal subgroups of $\bbg(k)$. For the precise statements, we refer the reader to the nice survey of these problems by Prasad and Rapinchuk~\cite{PRsur}. We should add that MP holds for $k$-isotropic groups and inner forms of type A~\cite{PRsur} and CSP holds for $k$-isotropic groups, by works of Raghunathan~\cite{Rag1,Rag2}. On the other hand, by a result of G.~Harder~\cite{Ha}, if $k$ is a global function field and $\bbg$ is anisotropic over $k$, then it is of type A. Hence, assuming $k$ is a global function field, MP is true for all absolutely almost simple $k$-groups except possibly for an anisotropic outer form of type A and CSP is true for all absolutely almost simple $k$-groups except possibly for an anisotropic form of type A.
\\

\noindent
In particular, by the above discussion, all of our results are unconditional for groups of type B, C, D (except ${\rm D}_4$) and ${\rm G}_2$.

\subsection{Outline of the argument}
In order to prove the upper bound of Theorem~\ref{t:maximal}, we essentially follow Borel-Prasad~\cite{BP}.  However here we have to provide estimates for all the finiteness results required for our quantitative statement. We first use Rohlfs' maximality criterion to get a description of maximal lattices in $G=\bbg_0(K)$. It essentially says we have to understand the following parameters:
\begin{itemize}
\item[1-] A function field $k$ and a place $\pfr_0$ over $k$, such that $k_{\pfr_0}\simeq K$.
\item[2-] A simply connected absolutely almost simple $k$-group $\bbg$, such that $\bbg\simeq\bbg_0$ over $k_{\pfr_0}$ ($K$ is identified with $k_{\pfr_0}$.)
\item[3-] A coherent family of parahoric subgroups $\{P_{\pfr}\}$ for any $\pfr\neq\pfr_0$.
\end{itemize}
Indeed as part of the criterion, we have that $\Gamma=N_G(\Lambda)$, where
\[
\Lambda=\bbg(k)\cap\prod_{\pfr\in V^{\circ}_k} P_{\pfr}.
\]
 Moreover, we can start with $\gcal$ the unique quasi-split $k$-inner form of $\bbg$ and parameterize  $\bbg$, more or less,  via elements of $H^1(k,\overline{\gcal})$, where $\overline{\gcal}$ is the adjoint form of $\gcal$. Furthermore there is a field extension $l$  of degree at most $3$ over $k$, over which $\gcal$ splits, and for a given $l$ and $k$ there is a unique quasi-split $k$-group of a given type which splits over $l$. 
 \\
 
 \noindent
 On the other hand, instead of giving coherent families of parahoric subgroups, first we will determine possible types of such families up to isomorphisms of the local Dynkin diagrams, and then give an upper bound on the number of admissible coherent families having the same type up to isomorphisms of the local Dynkin diagrams. Overall we have to estimate the number of:
\begin{itemize}
\item[1-] possible function fields $k$ and $l$ and a place $\pfr_0$ over $k$ with the above properties. (This give us a unique $\gcal$ in the above notation.)
\item[2-] possible elements of $H^1(k,\overline{\gcal})$. (This give us $\bbg$.)
\item[3-] possible coherent types, up to an isomorphism of the local Dynkin diagram at each place. (This gives us $\{P_{\pfr}\}$ a coherent family of parahoric subgroups of this given type, up to action of $\adbbg(\bba_k)$.)
\item[4-] elements of $\Cl(\overline{\bbg},\{\overline{P}_{\pfr}\})$. (Altogether we get $\Lambda$ and therefore $\Gamma$.)
\end{itemize}

\noindent
To get the lower bound of Theorem~\ref{t:maximal}, we shall appeal to results of Margulis and Prasad on the abstract isomorphisms between two lattices in a semisimple Lie group over a positive characteristic local field. 
\\

\noindent
To prove the upper bound of Theorem~\ref{t:CountingLattices}, as the number of maximal lattices is relatively small, it is enough to understand $s_x({\Gamma})$ the number of subgroups of index at most $x$ in $\Gamma$, where $\Gamma$ is a maximal lattice with covolume at most $x$. Then we will go through the following steps:
\begin{itemize}
\item[1-] Using Rohlfs' short exact sequence and previous estimates, we show that $\Gamma/\Lambda\le x^{c_1}$, where $c_1$ only depends on $G$.
\item[2-] Following Lubotzky~\cite{Lu}, we use a result of Babai, Cameron and Palfy to reduce the problem to $s_x({\prod_{\pfr\in T}P^{(1)}_{\pfr}})$, where $P^{(1)}_{\pfr}$ is the first congruence subgroup of $P_{\pfr}$ and $\deg T=\sum_{\pfr\in T}\deg \pfr\le \log x$.
\item[3-] We consider $\lcal$ the graded Lie algebra associated to the filtrations of the parahoric subgroups $\{P_{\pfr}\}_{\pfr\in T}$, similar to Lubotzky and Shalev~\cite{LSa} (also see~\cite[Section 6.2]{LS} or~\cite{ANS}) and reduce the problem to the following statement:
\[
\cdim_{\cal}[\cal,\cal]\le c_2\deg T+c_3\h\cdim_{\lcal} \cal,
\] 
where $\cal$ is an $\f_p$-subalgebra of $\lcal$, and $c_2$ and $c_3$ just depend on $\bbg_0$ and $K$.
\item[4-] Finally we deduce the above statement from Theorem~\ref{t:GradedLie}. Using a result of ~\cite{PR}, we view $\lcal\otimes_{\f_k}\F$ as the direct sum of graded algebras of certain parahoric subgroups of $\bbg(\widehat{k}_{\pfr})$, and then change these parahoric subgroups to the ``largest" parahoric subgroups, where we have a very good understanding of the graded Lie algebras and get the desired result.
\end{itemize}

\noindent
To get the lower bound, we essentially follow~\cite{Sh}. However we have to be extra careful as we are counting up to the action of $\Aut(G)$, and so we again appeal to the results of Margulis and Prasad on the abstract isomorphism of lattices in $G$.
\\

\noindent
At the end, we point out that the same arguments, not only gives us Theorem~\ref{t:SubgroupGrowth}, but also its uniform version, namely Corollary~\ref{c:SubgroupGrowth}.
\subsection{Acknowledgment.} I would like to thank Professor A. Lubotzky for introducing me to these problems. I am also very grateful for all the conversations that I had with him both from math and presentation points of view. I am in  Professor G. Prasad's debt for very helpful correspondences. Special thanks are due to Professor B. Conrad for providing the proof of Theorem~\ref{t:cohom}, on ``Shapiro's lemma" for flat cohomologies. I also thank Professor J. Ellenberg and Professor M. Belolipetsky for pointing out some minor mistakes. I also would like to thank anonymous referees for their valuable comments that have improved the quality of the initial 
manuscript.

\section{Notation, conventions and preliminaries.} 
 
\subsection{Field related notations.}\label{ss:fields}
In this paper, $k$ is a global function field. Let $V_k$ be the set of places of $k$. For any $\pfr\in V_k$, let $k_{\pfr}$ be the $\pfr$-adic completion of $k$, $\f_{\pfr}$ its residue field. Let $\bba_k$ be the ring of adeles of $k$. For any non-archimedean local field $K$, let $\widehat{K}$ be the maximal unramified extension of $K$, $\f$ residue field of $K$, and $\F$ the residue field of $\widehat{K}$. 

Let $\bbg_0$ be a simply connected, absolutely almost simple $K$-group of $K$-rank at least 2. Let $G=\bbg_0(K)$.  

For any finite set $X$, the cardinality of $X$ is denoted by either $|X|$ or $\# X$. If $H_1$ is a subgroup of $H_2$, the index of $H_1$ in $H_2$ is denoted by $[H_2:H_1]$.


\subsection{Flat and Galois cohomology.}
Let $E$ be a field and $\bbh$ an affine algebraic group-scheme over $E$. In this article, $H^1(E,\bbh)$ denotes the flat cohomology $H^1(\Spec(E),\bbh)$ and, similarly, $H^i(E,\bbh)$ denotes the $i^{th}$ flat cohomology if $\bbh$ is an abelian $E$-group-scheme. 
\\

\noindent
Let us summarize theorems on flat cohomology that will be used in the course of this article. 
\begin{thm}\label{t:cohom}
\begin{itemize}
\item[1-] If $\bbh$ is a smooth $E$-group-scheme, then the flat cohomology is canonically isomorphic to the Galois cohomology.
\item[2-] If there is a short exact sequence of abelian $E$-group schemes, then one functorially gets a long exact sequence of flat cohomologies. 
\item[3-] If $\bbh$ is an abelian $k$-group-scheme and $l$ is a finite separable extension of $k$ a global field, then naturally
\[
H^i(k_{\pfr},R_{l/k}(\bbh))\simeq \oplus_{\pfrl |\pfr} H^i(l_{\pfrl},\bbh),
\]
for any $\pfr\in V_k$ and $i\ge 1$. 
\end{itemize}
\end{thm}
\begin{proof} (We would like to thank B.~Conrad for providing the proof of the third part.)
The first part is proved in~\cite[Chapter I\!I\!I, Theorem 2.10, Theorem 3.9, Theorem 4.7 and Theorem 4.8(a)]{Mi}. The second one is part of the delta-functor structure of cohomology. (It is worth mentioning that starting with a short exact sequence with non-commutative $E$-group-schemes, one still gets a long exact sequence involving cohomology of degree zero and one~\cite[Chapter I\!I\!I, proposition 4.5]{Mi} or~\cite{Shat}).  The following proof of the third part was provided to us by B.~Conrad. Since the Weil restriction is pushforward on fppf abelian sheaves, this isomorphism comes from the degenerate Leray spectral sequence \cite[Chapter I\!I\!I, Theorem 1.18(a)]{Mi} for fppf cohomology relative to the finite etale covering $f:\Spec(l\otimes_k k_{\pfr}) \rightarrow \Spec(k_{\pfr})$.  The degeneration is due to the fact that the Weil restriction functor $f_*$ on fppf abelian sheaves is exact (and hence has vanishing higher derived functors) since etale-locally (and hence fppf-locally) on $\Spec(k_{\pfr})$ it becomes a totally split covering, for which exactness is obvious.
\end{proof}
\section{Variations on the results of Prasad and Raghunathan.}
This section has a threefold purpose:
\begin{enumerate}
\item We recall the construction of the graded algberas associated with parahoric subgroups over either a local field or its maximal unramified extension. (We follow Prasad and Raghunathan\rq{}s treatment~\cite{PR}.)
\item Using~\cite{PR}, we make a connection between the graded algebras associated with various parahoric subgroups (see Corollary \ref{c:ChangeParahorics}).
\item As another corollary of the results of ~\cite{PR}, we give the precise structure of the graded algebra associated with the ``largest\rq\rq{} parahoric subgroup (see Corollary \ref{c:SpecialAlgebra}). 
\end{enumerate}
 As it was pointed out in the introduction, these results play a crucial role 
in finding the rate of growth of $\rho_x(G)$ and the subgroup growth of any lattice in $G$ (see Section~\ref{ss:ReductionLieAlgebra}). As the 
nature of this section is completely different from the rest of the article, reader can 
easily skip it and return to the mentioned corollaries whenever needed.


\subsection{Root system related notations.}\label{ss:roots}
Let $\bbg$ be an absolutely simple, simply connected algebraic group defined over $K$. Let $\bbt$ be a maximal $\widehat{K}$-split torus defined over $K$, such a torus exists as part of Bruhat-Tits theory~\cite{T} . Let $\widehat{\bbt}$ be the centralizer of $\bbt$ which is also a torus defined over $K$ as $\bbg$ is quasi-split over $\widehat{K}$ and $\bbt$ is defined over $K$. Let $\Phi=\Phi(\bbt)$ be the $\widehat{K}$-root system of $\bbg$ with respect to $\bbt$, $\bbb$ a Borel subgroup defined over $\widehat{K}$ containing the centralizer of $\bbt$, $\Phi^+$ the set of positive roots in $\Phi$ with respect to this ordering, $\Pi$ the basis with respect to this ordering, $\Phi^{\bullet}$ the set of non-divisible roots of $\Phi$, $\Phi^{\bullet\bullet}$ the set of non-multipliable roots in $\Phi$. Let $\widehat{L}$ be the smallest Galois extension of $\widehat{K}$ over which $\widehat{\bbt}$ splits. Galois group of $\widehat{L}$ over $\widehat{K}$ acts on $X^*(\widehat{\bbt})$ the group of characters of $\widehat{\bbt}$, and it is well-known that there are correspondences between $Gal(\widehat{L}/\widehat{K})$-orbits of $\Phi(\widehat{\bbt})$, restriction of these roots to $\bbt$, and $\Phi$. In particular, for any $\phi\in\Phi$, there is $\widehat{\phi}$ a root in $\Phi(\widehat{\bbt})$ whose restriction to $\bbt$ is $\phi$. Let $\widehat{K}\subseteq\widehat{L}_{\phi}$ be the smallest subfield of $\widehat{L}$ such that $\widehat{\phi}$ is defined over $\widehat{L}_{\phi}$. Let $r={\rm lcm}_{\phi\in\Phi}[\widehat{L}_{\phi}:\widehat{K}]$ and $r_{\phi}=r/[\widehat{L}_{\phi}:\widehat{K}]$.


\subsection{Affine functions and inner product.}\label{ss:innerproduct}
Let $V=X_*(\bbt)\otimes_{\bbz}\bbr$ and $V^*=X^*(\bbt)\otimes_{\bbz}\bbr$ be its dual. Let $\langle\h,\h\!\rangle$ be a Weyl group-invariant positive definite inner product on $V^*$ such that in case the root system $\Phi$ is reduced and contains roots of unequal lengths, any short root has length $\sqrt{2}$ (we are using the same description of inner product as in~\cite{PR}). Let $\langle\h,\h\!\rangle$ also denote its extension to $V^*\times\bbr$ the space of all affine functions on $V$ (inner product of two affine functions is the inner product of their gradient  i.e. the $V^*$-component).  


\subsection{Explicit absolute affine roots.}\label{ss:AffineRoots}
Let $\Psi\subset V^*\times\bbr$ be the set of (absolute) affine roots of $\bbg$ relative to $\bbt$. 
 If $\bbg$ splits over $\widehat{K}$, then $\Psi=\Phi\times\bbz$. If $\bbg$ does not split over $\widehat{K}$ but the $\widehat{K}$-root system $\Phi$ is reduced, then $\Psi^{\vee}=\Phi^{\vee}\times\bbz$, where $\Psi^{\vee}=\{\frac{2\psi}{\langle\psi,\psi\rangle}\h|\h\psi\in\Psi\}$ and $\Phi^{\vee}=\{\frac{2\phi}{\langle\phi,\phi\rangle}\h|\h\phi\in\Phi\}$. If the $\widehat{K}$-root system is non-reduced, then $\bbg/\widehat{K}$ is an outer form of type ${\rm A}_n$, where $n$ is even. In this case, 
\[
\Psi=\{(\phi,n)\h|\h n\in\bbz, \phi\in\Phi^{\bullet}\}\cup\{(\phi,2n+1)\h|\h n\in\bbz,\phi\not\in\Phi^{\bullet}\}.
\]
We refer the reader to~\cite[Section 2.8]{PR} for details and further discussions.

\subsection{Congruence subgroups of root groups.}\label{ss:RootGroups}
For any $\widehat{K}$-root $\phi$, we let $\mathbb{U}_{\phi}$ be the corresponding root subgroup. There is a natural filtration on the $\widehat{K}$-points of this group, coming from the discrete valuation of $\widehat{K}$. Using this filtration, for any affine function $\psi=(\phi,s)$, one can define $U_{\psi}$ (sometimes denoted by $U_{\phi,s}$) a subgroup of $\mathbb{U}_{\phi}(\widehat{K})$. We refer the reader to either~\cite[Section 2.3]{PR} or \cite[Section 2.4]{MP} for details. Let $\bbg_{\phi}$ be the group generated by $\mathbb{U}_{\phi}$ and $\mathbb{U}_{-\phi}$, $\widehat{\bbt}_{\phi}=\bbg_{\phi}\cap\widehat{\bbt}$, $\widehat{T}^{\phi}_{0}$ the maximal bounded subgroup of $\widehat{\bbt}_{\phi}(\widehat{K})$, $\widehat{T}_0$ the maximal bounded subgroup of $\widehat{\bbt}(\widehat{K})$, and for any positive integer s, $\widehat{T}^{\phi}_{s\delta}$ the congruence subgroups of $\widehat{\bbt}(\widehat{K})$, where $\delta=(0,1)$ is the constant function. Again we refer the reader to \cite[Section 2.6]{PR} for the precise description of $\widehat{T}^{\phi}_{s\delta}$. Finally let $\widehat{T}_{s\delta}$ be the group generated by $\widehat{T}^{\phi}_{s\delta}$'s for $\phi\in\Phi$.


\subsection{Iwahori subgroup and explicit absolute affine basis.}\label{ss:IwahoriBasis}
Let $\widehat{I}$ be an Iwahori subgroup of $\bbg(\widehat{K})$ which is stable under the action of the Galois group of $\widehat{K}$ over $K$, and $I=\bbg(K)\cap\widehat{I}$. By changing the Borel subgroup $\bbb$ if needed we can assume that the product mapping
\[
\prod_{\phi\in\Phi^{\bullet}\cap\Phi^+}U_{\phi,0}\times\prod_{\phi\in\Phi^{\bullet}\cap\Phi^-}U_{\phi,r'_{\phi}}\times\widehat{T}_0\rightarrow \widehat{I}
\]
is bijective for every ordering of the factors of the product, where $r'_{\phi}$ is the smallest number such that $U_{\phi,r'_\phi}$ is a proper subgroup of $U_{\phi,0}$.
\noindent
Having the Iwahori subgroup, we get an ordering on $\Psi$ the affine root system and a basis $\Delta$.  
The local index of $\bbg/K$ consists of the Dynkin diagram of $\Delta$, together with the action of the Galois group of $\widehat{K}$ over $K$. If $\bbg$ splits over $\widehat{K}$,  
\[
\Delta=\{(\alpha,0)\h|\h\alpha\in\Pi\}\cup\{(-\rho,1)\},
\]
where $\rho$ is the highest root in $\Phi$. In this case, let $\psi_s=(-\rho,1)$. If $\bbg$ does not split over $\widehat{K}$ but its $\widehat{K}$-root system is not reduced, then
\[
\Delta=\{(\alpha,0)\h|\h\alpha\in\Pi\}\cup\{(-\rho_m,1)\},
\]
where $\rho_m$ is the dominant short root in $\Phi$. In this case, let $\psi_s=(-\rho_m,1)$. If $\Phi$ is non-reduced, then 
\[
\Delta=\{(\alpha,0)\h|\h\alpha\in\Pi\}\cup\{(-2\phi_m,1)\},
\]
where $\phi_m$ is the unique multipliable root in $\Pi$. In this case, let $\psi_s=(-2\phi_m,1)$.


\subsection{Standard parahoric subgroups.}\label{ss:Parahoric}
For any $\Xi\subseteq\Delta$, let $\widehat{P}_{\Xi}$ be the associated standard parahoric, i.e. the subgroup of $\bbg(\widehat{K})$ which is generated by $\widehat{I}$ and $U_{\alpha}$ for any $\alpha \in \Delta\setminus\Xi$. If $\widehat{P}_{\Xi}$ is invariant under the action of $Gal(\widehat{K}/K)$ the Galois group, then it is said to be defined over $K$ and we denote $\bbg(K)\cap\widehat{P}_{\Xi}$ the set of its $K$-rational points by $P_{\Xi}$. This happens if and only if $\Xi$ is invariant under the action of $Gal(\widehat{K}/K)$. Let us recall that any parahoric subgroup is conjugate to one and only one of these standard parahorics (one can consider this as a definition of a parahoric subgroup). A parahoric subgroup of $\bbg(\widehat{K})$ (resp. $\bbg(K)$) is called of type $\Xi\subseteq \Delta$ if it is conjugate to $\widehat{P}_{\Xi}$ (resp. $P_{\Xi}$).

\subsection{Filtrations of parahoric subgroups.}\label{ss:Filtration}
Following notations of \cite{PR}, let $m_{\alpha}$ be the uniquely determined positive integers such that
\[
\sum_{\alpha\in\Delta}m_{\alpha}\alpha=\delta,
\]
where $\delta$ is the constant function $(0,1)$. For any $\Xi\subseteq\Delta$ and affine function $\psi$, let
\[
l_{\Xi}(\psi)=\sum_{\alpha\in\Xi} t_{\alpha},
\] 
where $\psi=\sum_{\alpha\in\Delta}t_{\alpha}\alpha.$ Using $l_{\Xi}$, we can define congruence subgroups of $\widehat{P}=\widehat{P}_{\Xi}$ (resp. $P=P_{\Xi}$). More precisely, let $\widehat{P}_t$ be the subgroup of $\widehat{P}$ generated by $U_{\psi}$ with $l_{\Xi}(\psi)\ge t$ and $\widehat{T}_{s\delta}$ where $s$ is the smallest integer greater than or equal to $t/l_{\Xi}(\delta)$. If $\Xi$ is $Gal(\widehat{K}/K)$ invariant, then we set $P_t=\widehat{P}_t\cap\bbg(K)$.


\subsection{Associated graded Lie algebras.}\label{ss:AssociatedGraded}
Let $\widehat{P}=P_{\Xi}$ be a standard parahoric subgroup of $\bbg(\widehat{K})$. For any natural number $t$, let $\L^{\Xi}_t=\widehat{P}_t/\widehat{P}_{t+1}$. When there is no ambiguity, we simply write $\L_t$. For any natural number $t$, $\L_t$ is a finite dimensional $\F$-vector space. Let
\[
\L_{\Xi}=\oplus^{\infty}_{i=1}\L^{\Xi}_i,
\]
and consider it as a graded $\F$-Lie algebra via the following definition:
\[
[g_i\widehat{P}_{i+1},g_j\widehat{P}_{j+1}]:=(g_i,g_j)\widehat{P}_{i+j+1},
\]
for any natural numbers $i$, $j$, $g_i$ in $\widehat{P}_i$ and $g_j$ in $\widehat{P}_j$, where $(g,h)=ghg^{-1}h^{-1}$. When $\Xi$ is $Gal(\widehat{K}/K)$ invariant, we also consider the described filtration of $P=P_{\Xi}$. For any natural number $t$, let $L^{\Xi}_t=P_t/P_{t+1}$. Again if there is no ambiguity, we simply write $L_t$. One can view $L_t$ as a finite subgroup of $\L_t$. Indeed in this case,
$Gal(\widehat{K}/K)=Gal(\F/\f)$ acts semi-linearly on $\L_t$. Thus we get $\L_t(\f)$ an $\f$-structure on $\L_t$ and one can show that, in fact, $\L_t(\f)$ can be identified with $L_t$. Therefore the graded $\f$-Lie algebra
\[
L_{\Xi}=\oplus^{\infty}_{i=1} L^{\Xi}_i,
\]
can be considered as an $\f$-structure on $\L_{\Xi}$. We refer the reader to \cite[Sections 2.16,2.18,2.23,2.24]{PR}. G. Prasad and  M.~S.~Raghunathan~\cite[Section 2.19,2.20]{PR} give a ``Chevalley basis" of $\L_{\Xi}$. For any absolute affine root $\psi$, they introduce an element $X_ {\psi}$, and for a given $\beta$ in $\Phi$ and a natural number $s$ which is divisible by $r_{\beta}$, they give an element $H^{\beta}_{s\delta}$. They prove that
\begin{equation}\label{e:Basis}
\{X_{\psi}|\h l_{\Xi}(\psi)\ge 1\}\cup\{H^{\alpha}_{s\delta}|\h \alpha\in\Pi, s>0, r_{\alpha}|s\}
\end{equation}
is an $\F$-vector space basis of $\L_{\Xi}$. They also show that $\L^{\Xi}_t$ is spanned by
\begin{equation}\label{e:GradeBasis}
\{X_{\psi}|\h l_{\Xi}(\psi)= t\}\cup\{H^{\alpha}_{s\delta}|\h \alpha\in\Pi,\h  r_{\alpha}|s,\h s l_{\Xi}(\delta)=t\}.
\end{equation}
Further they give the following relations between the elements of this basis:
\begin{itemize}
\item[1-] Let $\psi, \eta\in\Psi$ with $l_{\Xi}(\psi), l_{\Xi}(\eta)\ge 1$, such that $\psi+\eta$ is not a constant. Then
\[
[X_{\psi},X_{\eta}]=0\hspace{1cm}{\rm if }\h\psi+\eta\not\in\Psi.
\]
\item[2-] Let $\psi,\eta\in\Psi$ such that $\psi-\eta\in\Psi$ and $l_{\Xi}(\psi-\eta),l_{\Xi}(\eta)\ge 1$. Then
\[
[X_{\eta},X_{\psi-\eta}]=\pm nX_{\psi},
\] 
where $n$ is the largest positive integer such that $\psi-n\eta$ is in $\Psi$.
\item[3-] Let $\psi=(\alpha,s), \eta=(-\alpha,s')\in\Psi$ and $l_{\Xi}(\psi), l_{\Xi}(\eta)\ge 1$. Then
\[
[X_{\psi},X_{\eta}]=\begin{cases}
				\pm H^{\alpha}_{(s+s')\delta}  &\text{if $\alpha\in\Phi^{\bullet\bullet}$,} \\
				\pm 2H^{\alpha}_{(s+s')\delta}&\text{if $\alpha\in\Phi\setminus\Phi^{\bullet\bullet}$.}
			      \end{cases}			 
\]
Moreover, for any positive integral multiple $s''$ of $r_{\alpha}$,
\[
[H^{\alpha}_{s''\delta},X_{\psi}]= \begin{cases}
				2 X_{\psi+s''\delta}  &\text{if $\alpha\in\Phi^{\bullet\bullet}$,} \\
				(2-(-1)^{s''}) X_{\psi+s''\delta}&\text{if $\alpha\in\Phi\setminus\Phi^{\bullet\bullet}$.}
			      \end{cases}	
\]  
\item[4-] Let $\psi=(\alpha,s)\in \Psi$ and $l_{\Xi}(\psi)\ge 1$, $\beta\in\Phi$, $\beta\neq\pm\alpha$, and $r_{\beta}|s$. Then $[H^{\beta}_{s\delta},X_{\psi}]$ is an integral multiple of $X_{\psi+s\delta}$. Moreover, in case $r|{\rm char}(\f)r_{\alpha}s$, 
\[
[H^{\beta}_{s\delta},X_{\psi}]=\begin{cases}
\frac{2\langle\beta,\alpha\rangle}{\langle\beta,\beta\rangle}X_{\psi+s\delta}&\text{if $r_{\alpha}|s$ and $\beta\in\Phi^{\bullet\bullet}$,} \\
\frac{\langle\beta,\alpha\rangle}{\langle\beta,\beta\rangle}X_{\psi+s\delta}&\text{if $r_{\alpha}|s$ and $\beta\in\Phi\setminus\Phi^{\bullet\bullet}$,}\\
0&\text{if $r_{\alpha}\nmid s$.}
\end{cases}
\] 
\item[5-] For any $\beta\in\Phi$, $H^{\beta}_{s\delta}=-H^{-\beta}_{s\delta}$, and $H^{\beta}_{2s\delta}=H^{2\beta}_{2s\delta}$ if $\beta$ is a multipliable root, and all elements of the form $H^{\beta}_{s\delta}$ commute with each other.
\item[6-] Let $\beta=\sum_{\alpha\in\Pi} n_{\alpha}\alpha\in\Phi$. Then in case $r|s$,
\[
H^{\beta}_{s\delta}=\begin{cases}
\sum_{\alpha\in\Pi} \frac{n_{\alpha}\langle\alpha,\alpha\rangle}{\langle\beta,\beta\rangle}(1+\sigma(\alpha))H^{\alpha}_{s\delta}&\text{if $\beta\in\Phi^{\bullet\bullet}$},\\
\sum_{\alpha\in\Pi} \frac{n_{\alpha}\langle\alpha,\alpha\rangle}{2\langle\beta,\beta\rangle}(1+\sigma(\alpha))H^{\alpha}_{s\delta}&\text{otherwise.}
\end{cases}
\] 
where $\sigma(\alpha)$ is  the $\bbr$-valued linear functional on $V^*$ which is identically zero in case $\Phi$ is reduced; otherwise, it takes 1 at the unique multipliable root in $\Pi$ and 0 at all the other elements of $\Pi$.
\item[7-] If $r\nmid s$ and $\beta=\sum_{\alpha\in\Pi}n_{\alpha}\alpha$ is a short positive root, then
\[
H^{\beta}_{s\delta}=(-1)^{rn(\beta)}\sum_{\alpha'\in\Phi'(\beta)}H^{\alpha'}_{s\delta},
\]
where $\Phi'(\beta)=\{\alpha'\in\Pi|\h\alpha' is\h short,\h r\nmid n_{\beta}\}$ and $n(\beta)=\sum_{\alpha short}n_{\alpha}$.
\end{itemize}
\begin{remark}
In number 4, in general, the integral coefficient just depends on $\alpha$, $\beta$ and $s\h\! l_{\Xi}(\delta)\h ({\rm mod}\h r)$. 
\end{remark}
\begin{cor}[Comparison]\label{c:ChangeParahorics}
Let $\Xi\subseteq\Delta$. $\L_{\Xi}$ can be naturally embedded into $\L_{\Delta}$ as a subalgebra of codimension at most $\dim \bbg$.
\end{cor}
\begin{proof}
Clearly $l_{\Xi}(\psi)\le l_{\Delta}(\psi)$ for any $\Xi\subseteq\Delta$ and $\psi\in\Psi$. Thus, by the above commutation relations, $\L_{\Xi}$ can be viewed as a subalgebra of $\L_{\Delta}$. The codimension assertions are direct consequence of the fact that there are at most $\dim \bbg$ many $\psi\in\Psi$ such that $l_{\Xi}(\psi)=0$. 
\\

\noindent
It should be also clarified that the choice of  signs in the commutative relations 2 and 3 are inherited from the group structure and so they are the same in both of the Lie algebras.
\end{proof}
\begin{remark}
First we remark that $\mathfrak{H}$, in the above argument, only lacks the first grade of $\L_{\Xi}$. Second we emphasize that this embedding is only at the level of  Lie algebras and not graded Lie algebras. Via this embedding, we change the grading, drastically. 
\end{remark}
\begin{cor}[Special graded Lie algebras]\label{c:SpecialAlgebra}
Let $\psi_s$ be as in \ref{ss:IwahoriBasis}. Then 
\begin{itemize}
\item[1-]$\L_{\{\psi_s\}}\simeq \gfr(\F)\otimes_{\F}t\h\F[t]$ if $\bbg$ splits over $\widehat{K}$,
\item[2-]$\L_{\{\psi_s\}}\simeq \oplus^{\infty}_{i=1} \gfr_{i ({\rm mod}\h r)} \otimes t^i $ if $\bbg$ does not split over $\widehat{K}$ and $\Phi$ is reduced,
\item[3-]$\L_{\{\psi_s\}}\simeq \oplus^{\infty}_{i=1} \gfr_{i ({\rm mod}\h 2)} \otimes t^i $ if $\Phi$ is non-reduced,
\end{itemize}
 where $\gfr(\F)=\gfr_0$ is the split simple Lie algebra of type $\Phi^{\bullet}$, and 
 $\widehat{\gfr}=
 \oplus^{r-1}_{i=0} \gfr_i
 $
is a perfect $\bbz/r\bbz$-graded algebra if ${\rm char}(\F)$ is bigger than all the entries of the Cartan matrix of $\Phi$.
\end{cor}
\begin{proof}
Because of the way, we chose $\psi_s$, one can see that $l_{\{\psi_s\}}((\phi,n))=n$; in particular, $l_{\{\psi_s\}}(\delta)=1$. First assume that $\Phi$ is reduced, and let 
\[
\{x_{\phi}|\phi\in\Phi\}\cup\{h_{\alpha}|\alpha\in\Pi\}
\]
be the Chevalley basis associated to the root system $\Phi$ and $\gfr$ the corresponding Lie algebra. Then by the above commutation relations and those of the Chevalley basis, one can easily see that the map
\[
 x_{\phi}\otimes t^n \mapsto X_{(\phi,n)}\hspace{1cm} \&\hspace{1cm} h_{\alpha}\otimes t^n \mapsto H^{\alpha}_{n\delta},
\]
where $r|n$, extends to a graded Lie algebra isomorphism between $\oplus^{\infty}_{n=1}\gfr(\F)\otimes t^{rn}$ and $\oplus^{\infty}_{n=1}\L^{\{\psi_s\}}_{rn}$. In particular, when $\bbg$ splits over $\widehat{K}$, $r=1$ and
\[
\L_{\{\psi_s\}}\simeq \gfr(\F)\otimes_{\F} t\F[t].
\]
If $\Phi$ is non-reduced, then by the above commutation relations and those of the Chevalley basis, we can see that the following map 
\[
\begin{cases}
x_{\phi}\otimes t^n \mapsto X_{(\phi,n)}& \text{if $2|n$,}\\
h_{\alpha} \otimes t^n \mapsto H^{\alpha}_{n\delta}&\text{if $2|n$ and $\alpha\in\Pi\cap\Phi^{\bullet\bullet}$,}\\
h_{\alpha}\otimes t^n \mapsto 2 H^{\alpha}_{n\delta}&\text{if $2|n$ and $\alpha\in\Pi\setminus\Phi^{\bullet\bullet}$,}
\end{cases}
\]
extends to a graded Lie algebra isomorphism between $\oplus^{\infty}_{n=1}\gfr(\F)\otimes t^{2n}$ and $\oplus^{\infty}_{n=1}\L^{\{\psi_s\}}_{2n}$, where $\gfr$ is the Chevalley Lie algebra of type $\Phi^{\bullet}$ and
\[
\{x_{\phi}|\phi\in\Phi^{\bullet}\}\cup\{h_{\alpha}|\alpha\in\Pi\}
\]
is a Chevalley basis of $\gfr$. (By our assumption, ${\rm char}(\F)\neq 2$.)
\\

\noindent
We notice that all the commutation relations depend on the gradient part of the affine roots and the constant parts modulo $r$. In particular, $\L^{\{\psi_s\}}_{i}$ can be naturally identified with $\L^{\{\psi_s\}}_{j}$ if $i\equiv j \pmod r$, and, via this identification, we get a $\gfr(\F)$-module structure on all of them. Indeed, by the same argument, we get $\widehat{\gfr}$, a $\bbz/r\bbz$-graded Lie algebra, such that 
\[
\L_{\{\psi_s\}}\simeq \oplus^{\infty}_{i=1}\gfr_{i ({\rm mod}\h r)}\otimes t^i.
\]
More precisely, if $\Phi$ is reduced, then let
\[
\{x_{(\phi,i)}|\phi\h short\h root\}\cup\{h_{(\alpha,i)}| \alpha\in\Pi, \alpha\h short\h root\},
\]
be a basis of $\gfr_i$ for $1\le i\le r-1$, $\gfr_0=\gfr(\F)$,
\[
\{x_{(\phi,0)}|\phi\in\Phi\}\cup\{h_{(\alpha,0)}| \alpha\in\Pi\}
\]
a Chevalley basis of $\gfr$ of type $\Phi$, and 
\[
\overline{\Psi}=\Phi\times\{0\}\cup\bigcup^{r-1}_{i=1}\{(\phi,i)|\phi\h short\h root\}. 
\]
If $\Phi$ is non-reduced, then let
\[
\{x_{(\phi,1)}|\phi\in\Phi\}\cup\{h_{(\alpha,1)}| \alpha\in\Pi\},
\]
be a basis of $\gfr_1$, $\gfr_0=\gfr(\F)$,
\[
\{x_{(\phi,0)}|\phi\in\Phi^{\bullet}\}\cup\{h_{(\alpha,0)}| \alpha\in\Pi\}
\]
a Chevalley basis of $\gfr$ of type $\Phi^{\bullet}$, and
\[
\overline{\Psi}=\{(\phi,0)|\phi\in\Phi^{\bullet}\}\cup\Phi\times\{1\}. 
\]
Then define the commutation relations between the elements of this chosen basis, by looking at those of $\L_{\{\psi_s\}}$ and modifying the constant parts of affine roots modulo $r$. (Since $r_{\phi}$'s are either 1 or $r$, there will not be any ambiguity.)
\\

\noindent
The isomorphisms are direct results of the way we defined $\widehat{\gfr}$, and  the perfectness of $\widehat{\gfr}$ is a consequence of the commutation relations coupled with our assumption on the characteristic of $\F$.
\end{proof}


\section{Counting maximal lattices in $G$.}

\subsection{Description of maximal lattices.}~\label{s:maximal}
Let $\Gamma$ be a maximal lattice in $G=\bbg_0(K)$. By Margulis' arithmeticity~\cite{Mar} and Rohlfs' maximality criteria~\cite[Proposition 2.9]{BP}, there are a function field $k$, $\pfr_0\in V_k$, a simply connected absolutely almost simple $k$-group $\bbg$, and a family of parahoric subgroups $\{P_{\pfr}\}$ of $\bbg(k_{\pfr})$ for any $\pfr\in V^{\circ}_k=V_k\setminus\{\pfr_0\}$ such that
\begin{itemize}
\item[1-] $k_{\pfr_0}\simeq K$.
\item[2-] $ \bbg \simeq \bbg_0$ over $K$ after identifying it with $k_{\pfr_0}$.
\item[3-] $\{P_{\pfr}\}$ is a coherent family of parahoric subgroups, i.e. $\bbg(k_{\pfr_0})\cdot\prod_{\pfr\in V^{\circ}_k}P_{\pfr}$ is an open subgroup of $\bbg(\bba_k)$.
\item[4-] $\Gamma=N_G(\Lambda)$, where $\Lambda=\bbg(k)\cap\prod_{\pfr\in V^{\circ}_k}P_{\pfr}$ is a principal congruence subgroup.
\item[5-] The following is a short exact sequence:
\[
1\rightarrow \mu(k_{\pfr_0})/\mu(k) \rightarrow \Gamma/\Lambda \rightarrow \delta(\bbg(k))_{\Theta^{\circ}}\rightarrow 1,
\] 
where $\mu$ is the center of $\bbg$, $\delta$ is the boundary map in the exact sequence
\[
1\rightarrow \mu(k)\rightarrow \bbg(k)\rightarrow \adg(k) \xrightarrow{\delta} H^1(k,\mu),
\]
and $\delta(\adg(k))_{\Theta^{\circ}}$ is the subgroup of $\delta(\adg(k))$ which preserves $\Theta^{\circ}=\{\Theta_{\pfr}\}_{\pfr\in V^{\circ}_k}$ the type of parahoric subgroups $P_{\pfr}$ (See ~\cite[Section 2]{BP}).
\end{itemize}


\subsection{Covolume of a principal congruence subgroup.}\label{s:Covolume}
Here we will recall the main result of G.~Prasad from~\cite{P}. The notations are the same as in \ref{s:maximal}. 
\\

\noindent
For any $\bbg$ and $k$ as above, there is a unique quasi-split $k$-group $\gcal$ which is an inner $k$-form of $\bbg$. Let $l$ be either a degree two or a degree three extension of $k$ over which $\gcal$ splits if it is not of type ${\rm D}_4^{(6)}$. It is a Galois extension of $k$ or the unique degree three extension of $k$ in a degree 6 Galois  extension of $k$, respectively when $\bbg$ is not a $k$ form of type ${\rm D}_4^{(6)}$ or it is. Because of the uniqueness  of $\gcal$, it is determined by its absolute type, $k$ and $l$. Let 
\[
\sfr=\sfr(\gcal)=
\begin{cases}
0 &\text{ if $\gcal$ is $k$-split,}\\
\frac{1}{2}(r-1)(r+2) &\text{if $\gcal$ is an outer form of type $\A_r$ with $r$ odd,}\\ 
\frac{1}{2}r(r+3) &\text{if $\gcal$ is an outer form of type $\A_r$ with $r$ even,}\\ 
2r-1& \text{if $\gcal$ is an outer form of type ${\rm D}_r$,}\\
26 & \text{if $\gcal$ is an outer form of tyep ${\rm E}_6$.}
\end{cases}
\]
Also let 
\[
B(\gcal)=q_k^{(g_k-1)\dim \gcal} \left(\frac{q_l^{g_l-1}}{q_k^{(g_k-1)[l:k]}}\right)^{\sfr}.
\]
For any $\pfr\in V_k$, let $\pcal_{\pfr}$ be a parahoric subgroup of $\gcal(k_{\pfr})$ chosen as in ~\cite[Section 1]{P}. Let us recall that these parahoric subgroups have maximum volume among all the parahoric subgroups of $\gcal(k_{\pfr})$, (as a consequence) always are special parahoric subgroups (and hyper-special whenever possible), and $\prod_{\pfr\in V_k} \pcal_{\pfr}$ is an open compact subgroup of $\gcal(\bba_k)$.
\\

\noindent
For any $\pfr\in V_k$, Bruhat-Tits theory provides us $\bbg_{\pfr}$ and $\gcal_{\pfr}$ two smooth affine group schemes  over $\o_{\pfr}$, such that
\begin{itemize}
\item[1-] The generic fibers of $\bbg_{\pfr}$ and $\gcal_{\pfr}$ are isomorphic to $\bbg$ and $\gcal$ over $k_{\pfr}$, respectively.
\item[2-] The $\o_{\pfr}$ points of $\bbg_{\pfr}$ and $\gcal_{\pfr}$ are isomorphic with $P_{\pfr}$ and $\pcal_{\pfr}$, respectively.
\end{itemize}

\noindent
Let $\lmbb_{\pfr}$ and $\lmcal_{\pfr}$ be a fixed maximal reductive subgroups of the special fibers of $\bbg_{\pfr}$ and $\gcal_{\pfr}$, respectively. Let $P_{\pfr_0}$ be a parahoric subgroup of $\bbg(k_{\pfr_0})$ with maximum volume among all parahoric subgroups, and $\vol$ be the Haar measure on $\bbg(k_{\pfr_0})$ such that $\vol(P_{\pfr_0})=1$. Let
\[
e(\pfr)=\frac{q_{\pfr}^{(\dim \lmbb_{\pfr}+\dim \lmcal_{\pfr})/2}}{\#\lmbb_{\pfr}(\f_{\pfr})},
e_{qs}(\pfr)=\frac{q_{\pfr}^{\dim \lmcal_{\pfr}}}{\#\lmcal_{\pfr}(\f_{\pfr})},\h{\rm and}\h\h
e'(\pfr)=\frac{e(\pfr)}{e_{qs}(\pfr)}.
\] 
Whenever $P_{\pfr}$ is a hyper-special  parahoric subgroup, $e'(\pfr)=1$. So for almost every $\pfr$, we have $e'(\pfr)=1$. Let $Z(\gcal)=\prod_{\pfr\in V_k} e_{qs}(\pfr)$. It is clear that $Z(\gcal)$ is larger than 1.  Prasad's main theorem in~\cite{P} says that 

\begin{thm}[\cite{P}]\label{t:covolume}
Following the above notations,
\[
\vol(G/\Lambda)=\tau_k(\bbg)B(\gcal)Z(\gcal)\prod_{\pfr\in V_k}e'(\pfr),
\]
where $\tau_k(\bbg)$ is the Tamagawa number of $\bbg/k$.
\end{thm}
\begin{remark}
A.~Weil conjectured that, if $\bbg$ is simply connected, absolutely almost simple, $k$-group, then $\tau_k(\bbg)=1$. This conjecture is proved in most of the cases. However it is still open for some twisted forms of type $\A$ and most of the exceptional types.
\end{remark}
\begin{remark}
In the rest of this article, we will assume that Weil's conjecture holds. Indeed, what we need is only a uniform lower bound for the Tamagawa numbers. 
\end{remark}


\subsection{Lower bound on $\vol(G/\Gamma)$ and local factors.}\label{ss:LowerBound}

In this section, we will summarize and adapt some of the results of A.~Borel and G.~Prasad from ~\cite{BP}, and then give a lower bound on the covolume of $\Gamma$. Following them, let 
\[
\vare=\vare(\bbg)=\begin{cases}
2 & \text{ if $\bbg$ is of type ${\rm D}_r$ with $r$ even,}\\
1 & \text{otherwise.}
\end{cases}
\]
Let $t$ be the exponent of $\mu$, i.e. 
\[
t=t(\bbg)=\begin{cases}
r+1 & \text{if $\bbg$ is of type ${\rm A}_r$,}\\
2  & \text{if $\bbg$ is of type ${\rm B}_r,{\rm C}_r$, ${\rm D}_r$ with $r$ even, or ${\rm E}_7$,}\\
3  & \text{if $\bbg$ is of type ${\rm E}_6$,}\\
4  & \text{if $\bbg$ is of type ${\rm D}_r$ with $r$ odd,}\\
1  & \text{if $\bbg$ is of type ${\rm E}_8, {\rm F}_4$ or ${\rm G}_2$.}
\end{cases}
\]
By the definition, $\vare$ and $t$ just depend on the absolute type of $\bbg$, and so, in particular, they are completely determined by $G$. 
\\

\noindent
For any $\pfr$, $H^1(k,\mu)$ acts on $\dcal_{\pfr}$ the local Dynkin diagram of $\bbg/k_{\pfr}$. Let $\xi_{\pfr}$ be the induced homomorphism to the group of isometries of the local Dynkin diagram and $\Xi_{\pfr}$ the subgroup of the image of $\xi_{\pfr}$ which preserves $\Theta_{\pfr}$. Let $\xi^{\circ}=(\xi_{\pfr})_{\pfr\in V_k^{\circ}}$, $\xi=(\xi_{\pfr})_{\pfr\in V_k}$, $H^1(k,\mu)_{\xi^{\circ}}=\ker \xi^{\circ}$, and $H^1(k,\mu)_{\xi}=\ker \xi$. So it is clear that
\[
\# \delta(\bbg(k))_{\Theta^{\circ}}\le \#H^1(k,\mu)_{\xi^{\circ}}\cdot\prod_{\pfr\in V_k^{\circ}} \#\Xi_{\pfr} \le
 \#H^1(k,\mu)_{\xi}\cdot\prod_{\pfr\in V_k} \#\Xi_{\pfr}\cdot \#{\rm Aut}\h \dcal_{\pfr_0}. 
\]
We note that $\#{\rm Aut}\h \dcal_{\pfr_0}$ just depends on $G$. 
\\

\noindent
So, by Theorem~\ref{t:covolume} and the short exact sequence given in ~\ref{s:maximal}, we have that
\[
\vol(G/\Gamma)\ge c_4 B(\gcal)\cdot t^{-\vare} (\#H^1(k,\mu)_{\xi})^{-1} \prod_{\pfr\in V_k}\frac{ e(\pfr)}{\#\Xi_{\pfr}}.
\]
On the other hand, by virtue of~\cite[Proposition 5.1, Proposition 5.6]{BP}, (note that $H^1(k,C)_{\xi}$ in ~\cite{BP} is denoted by $H^1(k,\mu)_{\xi^{\circ}}$ in our paper) we have
\begin{prop}\label{p:H1Upper}
In the above setting, 
\[
\# H^1(k,\mu)_{\xi} \le \begin{cases}
c_5 h_k^{\vare} t^{\vare\#T(\bbg)}& \text{if $\gcal$ splits over $k$,}\\
c_5 h_l 4^{\#T(\bbg)}\cdot \frac{q_l^{2(g_l-1)}}{q_k^{2(g_k-1)[l:k]}}& \text{if $\bbg$ is of type ${\rm D}_r$ with $r$ even,}\\
c_5 h_l t^{\#T(\bbg)}& \text{otherwise,}
\end{cases}
\]
where we assume $\gcal$ does not split over $k$ in the second and the third cases, $c_5$ is a constant depending on $G$,  
\[
T(\bbg)=\{\pfr\in V_k|\h\text{ $\bbg$ splits $/ \widehat{k}_{\pfr}$ and $\bbg$ is not quasi-split $/ k_{\pfr}$}\},
\]
and $h_k$ (resp. $h_l$) is the class number of $k$ (resp. $l$). 
\end{prop}
\noindent
From the above discussions, we have  
 \begin{prop}\label{p:CovolumeLowerBound}
 In the above setting,
 \[
\vol(G/\Gamma)\ge c_6\cdot q_k^{(g_k-1)\dim \gcal}\cdot h_l^{-\vare'}\cdot \left(\frac{q_l^{(g_l-1)}}{q_k^{(g_k-1)[l:k]}}\right)^{\sfr'(\gcal)}\prod_{\pfr\in V_k} f(\pfr),
 \]
 where $c_6$ just depends on $G$, $\vare'=1$ (resp. $\vare$) if $l=k$ (resp. otherwise), $\sfr'(\gcal)=\sfr(\gcal)-2$ (resp. $\sfr(\gcal)$) if $G$ is of type ${\rm D}_r$ with $r$ even (resp. otherwise), and 
\[
 f(\pfr)=
 \begin{cases}
  e(\pfr)/\#\Xi_{\pfr} & \text{if $\pfr\not\in T(\bbg)$},\\
  &\\
e(\pfr)/(\#\Xi_{\pfr}\cdot t^{\vare}) & \text{if $\pfr\in T(\bbg)$}.
 \end{cases}
 \]
 \end{prop}
\noindent
 Here, we would like to give a lower bound for the local factors associated to the ``bad" primes.
 \begin{prop}\label{p:LocalFactors}
 There is a positive constant $\sigma=\sigma(G)$ such that
 \[
 f(\pfr)\ge q_{\pfr}^{\sigma}
 \]
 if either $\bbg$ is not quasi-split over $k_{\pfr}$, $P_{\pfr}$ is not special, or $\bbg$ is quasi-split over $k_{\pfr}$, $\bbg$ splits over $\widehat{k}_{\pfr}$ and $P_{\pfr}$ is not hyper-special.
 \end{prop}
 \begin{proof}
 We will follow the proof given in the appendix C of~\cite{BP}, where Borel and Prasad essentially show that $f(\pfr)>1$, for any $\pfr$. 
 \\
 
 \noindent
 We start with the case where $\bbg$ is not quasi-split over $k_{\pfr}$.  For any such $\pfr$, let $P_{\pfr}^m$ be a parahoric subgroup of $\bbg(k_{\pfr})$ of maximum volume such that $P_{\pfr}\cap P_{\pfr}^m$ contains an Iwahori subgroup and whenever $P_{\pfr}$ is such a parahoric subgroup, we set them equal; then for almost all $\pfr$, $P_{\pfr}=P_{\pfr}^m$. It is proved in~\cite[Section 3]{BP} that 
 \[
(\#\Xi_{\pfr})^{-1} e(\pfr)\ge e_m(\pfr).
\]
 By~\cite[proposition 2.10]{P}, for any such $\pfr$, 
 \[
 e_m(\pfr)\ge \frac{q_{\pfr}^{r_{\pfr}+1}}{q_{\pfr}+1},
 \]
  where $r_{\pfr}$ is the $\widehat{k}_{\pfr}$-$\rank$ of $\gcal$.  Now as it has been pointed out in~\cite[Section 4.3]{Be}, the same proof as in the mentioned appendix gives us the claimed $\sigma$.
 \\
 
 \noindent
 Now, assume that $\bbg$ is quasi-split over $k_{\pfr}$, it is split over $\widehat{k}_{\pfr}$, and $P_{\pfr}$ is not hyper-special; then, by~\cite[proposition 2.10]{P}, 
 \[
 e(\pfr)\ge \frac{q_{\pfr}^{r_{\pfr}+1}}{q_{\pfr}+1}.
 \]
 Since $\bbg$ splits over $\widehat{k}_{\pfr}$ and  $\# \Xi_{\pfr} \le t^{\vare}$, the same argument works.\\

\noindent
Finally assume that $\bbg$ is quasi-split over $k_{\pfr}$ and not split over $\widehat{k}_{\pfr}$, and $P_{\pfr}$ is not special. In this case,  $\bbg/k_{\pfr}$ is a residually split group. By looking at the table of such groups in~\cite{T}, we see that local Dynkin diagrams of only two of them have a non-trivial automorphism, for both of which we have $\#\Xi_{\pfr}\le 2$ and $r_{\pfr}\ge 2$. On the other hand, again by \cite[proposition 2.10]{P}, we have
 \begin{equation}\label{e:inequality}
 e(\pfr)\ge \frac{q_{\pfr}^{r_{\pfr}+1}}{q_{\pfr}+1}.
 \end{equation}
 So either $\Xi_{\pfr}$ has a non-trivial element,  in which case by the above argument and a similar reasoning as in the previous cases we get the desired $\sigma$, or not, in which case getting $\sigma$ is a straightforward conclusion of~(\ref{e:inequality}).
 \end{proof}
 \begin{definition}\label{d:Ramified}
 Let $\mathfrak{R}(\Gamma):=T_c\cup T_l$ be the set of {\em ramified primes} of $\Gamma$, where $T_c$ and $T_l$ are defined as follows:
 \begin{itemize}
 \item[1-](Ramified at the level of commensurability) This is the set of $\pfr\in V_k$ such that $\bbg$ is not quasi-split over $k_{\pfr}$. We shall denote it by $T_c$. 
 \item[2-](Ramified at the local level) This set consists of either $\pfr\in V_k$ which is ramified over $l$ (alternatively $\gcal$ does not split over $k_{\pfr}$) and $P_{\pfr}$ is not special, or $\pfr\in V_k\setminus T_c$ and is not ramified over $l$ and $P_{\pfr}$ is not hyper-special. We shall denote it by $T_l$. 
 \end{itemize}
 \end{definition}
 

 \subsection{The main inequality.}
 Here we will use Riemann hypothesis for curves over finite fields proved by A.~Weil to estimate $h_l$ and couple it with an estimate of $g_k$ in terms of $g_l$ to prove the following inequality. 
\begin{thm}[Main Inequality]\label{t:MainInequality}
In the above setting, there are positive numbers $c, \sigma, \sigma_1$ and $\sigma_2$ depending only on $G$ such that
\[
\vol(G/\Gamma)\ge c\cdot q_k^{\sigma_1 g_k+\sigma_2 g_l}\cdot \prod_{\pfr\in \mathfrak{R}(\Gamma)}q_{\pfr}^{\sigma}.
\]
\end{thm}
\begin{proof}
By Riemann hypothesis for curves over finite fields, we know that 
\begin{equation}\label{e:RH}
(\sqrt{q_l}-1)^{2g_l}\le h_l \le (\sqrt{q_l}+1)^{2g_l}.
\end{equation}
By Proposition~\ref{p:CovolumeLowerBound}, Proposition~\ref{p:LocalFactors} and inequality~(\ref{e:RH}), we have that
\[
\vol(G/\Gamma)\ge c_6\cdot (\sqrt{q_l}+1)^{-2g_l\vare'}\cdot q_k^{(g_k-1)\dim \gcal}\cdot \left(\frac{q_l^{(g_l-1)}}{q_k^{(g_k-1)[l:k]}}\right)^{\sfr'(\gcal)}\prod_{\pfr\in \mathfrak{R}(\Gamma)} q_{\pfr}^{\sigma}
\]
\begin{equation}\label{e:RawInequality}
=c_7\cdot (\sqrt{q_l}+1)^{-2g_l\vare'}\cdot q_k^{g_k(\dim \gcal-[l:k]\sfr')}\cdot q_l^{g_l\sfr'}\prod_{\pfr\in \mathfrak{R}(\Gamma)} q_{\pfr}^{\sigma},
\end{equation}
where $c_7$ just depends on $G$ as $q_k$ is at most equal to the number of elements of the residue field of $K$ and $\dim \gcal$ just depends on $G$.
\\

\noindent
We consider inner forms and outer forms separately. First assume that $k=l$; then, by inequality~(\ref{e:RawInequality}), we have
\[
\vol(G/\Gamma)\ge c_7 (\sqrt{q_k}+1)^{-2g_k}\cdot q_k^{g_k \dim \gcal}\prod_{\pfr\in \mathfrak{R}(\Gamma)} q_{\pfr}^{\sigma}. 
\]
Since we assumed that the characteristic of $K$ is at least $3$, $q_k$ is at least $3$. Hence $\sqrt{q_k}+1< q_k^{1-\sigma_3}$, for a positive number $\sigma_3$. Thus $\sigma_4=\dim\gcal-2(1-\sigma_3)$ is a positive number and 
\begin{equation}\label{e:k=l}
\vol(G/\Gamma)\ge c_7\cdot q_k^{\sigma_4 g_k}\prod_{\pfr\in \mathfrak{R}(\Gamma)} q_{\pfr}^{\sigma}. 
\end{equation}
Now assume that $q_l=q_k^{2}$; then $g_l=g_k$, and so by inequality~(\ref{e:RawInequality}), we have
\[
\vol(G/\Gamma)\ge c_7\cdot (q_k+1)^{-2g_k} q_k^{g_k \dim \gcal} \prod_{\pfr\in \mathfrak{R}(\Gamma)} q_{\pfr}^{\sigma}.
\] 
We notice that on one hand we have $q_k+1<q_k^{2-\sigma_5}$, for a fixed positive number $\sigma_5$ and any $q_k$, and on the other hand $\sigma_6=\dim \gcal- 2( 2-\sigma_5)$ is positive as the $K$-rank of $\bbg_0$ is at least $2$.  Therefore we have
\begin{equation}\label{e:unramified}
\vol(G/\Gamma)\ge c_7\cdot q_k^{\sigma_6 g_k} \prod_{\pfr\in \mathfrak{R}(\Gamma)} q_{\pfr}^{\sigma}. 
\end{equation}
Now assume that $q=q_l=q_k$; then, by inequality~(\ref{e:RawInequality}) coupled with the fact that $\sqrt{q}+1\le q$ as $q\ge 3$, we have
\begin{equation}\label{e:k_not_l}
\vol(G/\Gamma)\ge c_7\cdot q^{g_k (\dim\gcal-[l:k]\sfr')+g_l(\sfr'-2\vare')} \prod_{\pfr\in \mathfrak{R}(\Gamma)} q_{\pfr}^{\sigma}.
\end{equation}
In the following table, we give the possible values of $\sigma_7=\dim\gcal-[l:k]\sfr'$ and $\sigma_8=\sfr'-2\vare'$ for all the possible types. 
\[
\begin{array}{c|c|c|c|c|c}
&\dim \gcal&\sfr &\sfr' &\dim\gcal-[l:k]\sfr' &\sfr'-2\vare'\\
\hline
{\rm A}_r^{(2)}, 2|r&r(r+2)& \frac{r(r+3)}{2}& \frac{r(r+3)}{2}& -r&\frac{(r+4)(r-1)}{2} \\
\hline

{\rm A}_r^{(2)}, 2\nmid r& r(r+2)& \frac{(r-1)(r+2)}{2}& \frac{(r-1)(r+2)}{2}&r+2& \frac{(r+3)(r-2)}{2}\\
\hline

{\rm D}_r^{(2)}, 2|r& r(2r-1)& 2r-1& 2r-3&(2r-1)(r-2)+4&2r-7\\
\hline

{\rm D}_r^{(2)}, 2\nmid r&r(2r-1)& 2r-1& 2r-1&(2r-1)(r-2)& 2r-3 \\
\hline

{\rm E}_6^{(2)}&78&26&26& 26&24\\
\hline

{\rm D}_4^{(3),(6)}& 28 &7&5&13&1
\end{array}
\]
We observe that both of these values are positive except when $\bbg$ is of type ${\rm A}_r$ with $r$ even. When $\bbg$ is of this type, we have
\begin{equation}\label{e:A_r}
\begin{array}{rc}
\vol(G/\Gamma)\ge& c_7\cdot q^{-g_k r +g_l (r+4)(r-1)/2}\prod_{\pfr\in \mathfrak{R}(\Gamma)} q_{\pfr}^{\sigma}\\
&\\
\ge& c_8\cdot q^{[-(g_k-1) r+(g_l-1) r/2]+(g_l-1)[(r+4)(r-1)-r]/2 }\prod_{\pfr\in \mathfrak{R}(\Gamma)} q_{\pfr}^{\sigma}\\
&\\
\ge&  c_9\cdot q^{g_l[(r+1)^2-5]/2 }\prod_{\pfr\in \mathfrak{R}(\Gamma)} q_{\pfr}^{\sigma},
\end{array}
\end{equation}
where $c_8$ and $c_9$ just depend on $G$ as $(g_l-1)-[l:k](g_k-1)$ is non-negative and $q$ is at most equal to the number of elements of the residue field of $K$. We also note that $\sigma_9=[(r+1)^2-5]/2$ is positive since $r$ is even and positive. 
\\

\noindent
The only remaining case is when $\gcal$ is of type ${\rm D}_4^{(3)}$ and $q_l=q_k^3$, in which case $g_l=g_k$ and, by inequality~(\ref{e:RawInequality}), we have
\[
\vol(G/\Gamma)\ge c_7\cdot (\sqrt{q_k}^3+1)^{-4 g_k}\cdot q_k^{28(g_k-1)}\prod_{\pfr\in \mathfrak{R}(\Gamma)} q_{\pfr}^{\sigma}.
\]
We note that $\sqrt{q_k}^3+1\le q_k^2$ for any $q_k\ge 3$. Hence we have
\begin{equation}\label{e:D_4}
\vol(G/\Gamma)\ge c_{10}\cdot q_k^{20 g_k}\prod_{\pfr\in \mathfrak{R}(\Gamma)} q_{\pfr}^{\sigma},
\end{equation}
where $c_{10}$ just depends on $G$.
\\

\noindent 
It is straightforward to finish the proof using inequalities~(\ref{e:k=l}), (\ref{e:unramified}), (\ref{e:k_not_l}), (\ref{e:A_r}) and (\ref{e:D_4}).
\end{proof}
\subsection{Number of possible pair of function fields $(k,l)$.}\label{ss:k-l}
Here we will use a result of de Jong and Katz~\cite{dJK} on the number of curves over  a finite field to give an upper bound on the number of possible pairs of function fields $(k,l)$.
\\

\noindent
First we point out that since $k_{\pfr_0}$ is isomorphic to $K$ for some $\pfr_0\in V_k$, $q_k^{\deg \pfr_0}=\#\f$ where $\f$ is the residue field of $K$. Hence the number of possibilities for $q_k$ is bounded only by $G$ and this upper bound is independent of $x$ the bound on the covolume of $\Gamma$. So without loss of generality we can assume that $q_k$ is fixed.
\\

\noindent
If $\vol(G/\Gamma)\le x$, then by Theorem~\ref{t:MainInequality}, we have
\[
x\ge c p^{\sigma_1 g_k+\sigma g_l}.
\]
Hence $g_k\le \sigma_1^{-1}\log(x/c)$ and $g_l\le \sigma_2^{-1}\log(x/c)$. By de Jong and Katz~\cite{dJK}, the number of function fields with a given constant field of size $q$ and genus $g$ is at most
$
c'^{g\log g},
$
where $c'=c'(q)$ just depends on $q$. Hence number of possible pair of function fields $(k,l)$ is at most
\[
x^{c_{11}\log\log x},
\]
where $c_{11}$ is a positive number which just depends on $G$.
\begin{remark}\label{r:dJK}
It is worth mentioning that by a similar argument as above one gets that the number of possible 
pairs of the function fields $(k,l)$  is at most $x^{c_{11}}$ (where $c_{11}$ is a number which 
just depends on $G$) if Question~\ref{q:Maximal} has a positive answer. 
\end{remark}
\subsection{Number of possible $\bbg$.}
Here we use local-global principle to give an upper bound for the number of admissible $\bbg$'s.
\\

\noindent
From now on we shall fix $(k,l)$ and as a consequence $\gcal$. Since $\bbg$ is a $k$-inner form of $\gcal$, in order to find an upper bound on the number of possible such groups, we can give an upper bound on the number of possible elements of $H^1(k,\adgcal)$. (We note that $\adgcal$ is a smooth $k$-group scheme, and so the Galois and the flat cohomologies are the same.) First we show that
\[
H^1(k,\adgcal)\rightarrow \prod_{\pfr\in V_k} H^1(k_{\pfr},\adgcal)
\]
is an injective map, and then we count at the local level. 
\\

\noindent
\begin{lem}\label{l:nthroot}
Let $\phi_n$ be the homomorphism $x\mapsto x^n$ from the multiplicative $k$-group scheme $\bbg_m$ to itself. Let $\mu_n$ be the $k$-group scheme kernel of $\phi_n$. Then
\[
H^2(k,\mu_n)\rightarrow \prod_{\pfr\in V_k} H^2(k_{\pfr},\mu_n)
\]
is an injective map.
\end{lem}
\begin{proof}
By the definition, we have the following short exact sequence
\[
1\rightarrow \mu_n \rightarrow \bbg_m \xrightarrow{\phi_n} \bbg_m \rightarrow 1.
\]
Therefore we get the following diagrams (each row is a long exact sequence)
\begin{equation}\label{e:ExactN}
\begin{array}{ccccc}
H^1(k,\bbg_m)&\rightarrow& H^2(k,\mu_n)&\rightarrow &H^2(k,\bbg_m)\\
\downarrow&&\downarrow&&\downarrow\\
\prod_{\pfr\in V_k}H^1(k_{\pfr},\bbg_m)&\rightarrow&\prod_{\pfr\in V_k}H^2(k_{\pfr},\mu_n)&\rightarrow&\prod_{\pfr\in V_k}H^2(k_{\pfr},\bbg_m).
\end{array}
\end{equation} 
Since $\bbg_m$ is a smooth $k$-group scheme, the flat cohomologies are isomorphic to the Galois cohomologies. Hence $H^1(k,\bbg_m)=1$, $H^1(k_{\pfr},\bbg_m)=1$, for  any $\pfr$, $H^2(k,\bbg_m)\simeq Br(k)$ and $H^2(k_{\pfr},\bbg_m)\simeq Br(k_{\pfr})$. By Brauer-Hasse-Noether theorem, we have that
\[
0\rightarrow Br(k) \rightarrow \oplus_{\pfr\in V_k} Br(k_{\pfr}) \rightarrow \bbq/\bbz \rightarrow 0
\]
is a short exact sequence. Thus the last vertical arrow in diagram (\ref{e:ExactN}) is injective. So, following the arrows in diagram~(\ref{e:ExactN}), one can easily  finish the argument.
\end{proof}
\begin{lem}\label{l:embeddingR}
Let $\mu_n$ be as above; then
\[
H^2(k,R_{l/k}(\mu_n))\rightarrow \prod_{\pfr\in V_k} H^2(k_{\pfr},R_{l/k}(\mu_n))
\]
is an injective map.
\end{lem}
\begin{proof}
By Shapiro's lemma~(see Theorem \ref{t:cohom}), we have 
\[
H^2(k,R_{l/k}(\mu_n))\simeq H^2(l,\mu_n)\h\&\h H^2(k_{\pfr},R_{l/k}(\mu_n))\simeq \oplus_{\pfrl |\pfr}H^2(l_{\pfrl},\mu_n).
\]
Hence the map in question is induced by the map
\[
H^2(l,\mu_n)\rightarrow \prod_{\pfrl\in V_l} H^2(l_{\pfrl},\mu_n),
\]
which is an embedding by Lemma~\ref{l:nthroot} and we are done.
\end{proof}
\begin{lem}\label{l:exactR}
Let $\mu_n$ be as above,  $\nu=R_{l/k}(\mu_n)$, and $\nu^{(1)}=R^{(1)}_{l/k}(\mu_n)$ be a $k$-group scheme which is the kernel of the norm map
$
\ker(R_{l/k}(\mu_n)\xrightarrow{N_{l/k}} \mu_n);
$
then
\[
\phi: H^2(k,\nu^{(1)})\rightarrow \prod_{\pfr\in V_k} H^2(k_{\pfr},\nu^{(1)})
\]
is an injective map if either $[l:k]$ and $n$ are coprime, or $[l:k]=2$.
\end{lem}
\begin{proof}
By the definition, the following  is a short exact sequence
\[
1\rightarrow \nu^{(1)} \rightarrow \nu \rightarrow \mu_n \rightarrow 1.
\]
Hence we get the following diagram (each row is an exact sequence):
\begin{equation}\label{e:ExactR}
\begin{array}{ccccc}
H^1(k,\mu_n)&\xrightarrow{\delta}& H^2(k,\nu^{(1)})&\rightarrow &H^2(k,\nu)\\
\downarrow&&\downarrow&&\downarrow\\
\prod_{\pfr\in V_k}H^1(k_{\pfr},\mu_n)&\xrightarrow{(\delta_{\pfr})}&\prod_{\pfr\in V_k}H^2(k_{\pfr},\nu^{(1)})&\rightarrow&\prod_{\pfr\in V_k}H^2(k_{\pfr},\nu).
\end{array}
\end{equation}
As all the group schemes are abelian, it is enough to show that the kernel of $\phi$ is trivial. Since, by Lemma~\ref{l:embeddingR}, the last vertical arrow is an embedding and the first row is an exact sequence,  kernel of $\phi$ is in the image of the boundary map $\delta$. On the other hand, $H^1(k,\mu_n)\simeq k^{\times}/{k^{\times}}^n$, $H^1(k_{\pfr},\mu_n)\simeq k_{\pfr}^{\times}/{k_{\pfr}^{\times}}^n$, and we have the following diagram (each row is an exact sequence):
\[
\begin{array}{ccccc}
l^{\times}/{l^{\times}}^n&\rightarrow & k^{\times}/{k^{\times}}^n &\xrightarrow{\delta}& H^2(k,\nu^{(1)})\\
\downarrow&&\downarrow&&\downarrow\\
\prod_{\pfr\in V_k} (l\otimes_k k_{\pfr})^{\times}/{ (l\otimes_k k_{\pfr})^{\times}}^n&\rightarrow &\prod_{\pfr\in V_k} k_{\pfr}^{\times}/{k_{\pfr}^{\times}}^n&\xrightarrow{(\delta_{\pfr})}&\prod_{\pfr\in V_k}H^2(k_{\pfr},\nu^{(1)})
\end{array}
\]
So overall, by the above discussion and following the arrows in diagram~\ref{e:ExactR} and the above one, for any $\eta\in \ker \phi$, one can find $x\in k^{\times}$ such that
\begin{itemize}
\item[1-] $\delta(x {k^{\times}}^n)=\eta$.
\item[2-] For any $\pfr\in V_k$, there is $y_{\pfr}\in \oplus_{\pfrl|\pfr} {l_{\pfrl}}^{\times}$ such that $x {k_{\pfr}^{\times}}^n=N_{l/k}(y_{\pfr}) {k_{\pfr}^{\times}}^n$.
\end{itemize}
If $[l:k]$ and $n$ are coprime, then for some $y\in l$, we have $x {k^{\times}}^n=N_{l/k}(y) {k^{\times}}^n$, and so $\eta$ is trivial, and we are done.
\\

\noindent
If they are not coprime, then, by the assumptions, $[l:k]=2$ and $n$ is even, in which case, ${k^{\times}}^n\subseteq N_{l/k}(l^{\times})$ and ${k_{\pfr}^{\times}}^n\subseteq N_{l/k}((l\otimes_k k_{\pfr})^{\times})$. Hence for any $\pfr$, $x$ is in the image of the norm map, i.e. $x\in N_{l/k}((l\otimes_k k_{\pfr})^{\times})$. Thus, by Hasse norm theorem~\cite[Chapter 10]{Sch}, there is $y\in l$ such that $N_{l/k}(y)=x$ and so $\eta$ is again trivial, and we are done.
\end{proof}
\begin{thm}\label{t:adjointembedd}
Let $\gcal, k$ and $l$ be as before; then
\[
\phi:H^1(k,\adgcal)\rightarrow \prod_{\pfr\in V_k} H^1(k_{\pfr},\adgcal)
\] 
is an injective map.
\end{thm}
\begin{proof}
Since $\bbg$ is a $k$-inner form of $\gcal$, their centers are $k$-isomorphic. Hence we have the following short exact sequence:
\[
1\rightarrow \mu\rightarrow \gcal \rightarrow \adgcal\rightarrow 1,
\]
and either $\mu$ is $k$-isomorphic to $\mu_t^{\vare}$ or $R^{(1)}_{l/k}(\mu_t)$, where $\vare=\vare(\bbg)$ and $t=t(\bbg)$. From the above short exact sequence, one can conclude the following diagram (each row is an exact sequence):
\begin{equation}\label{e:long}
\begin{array}{ccccc}
H^1(k,\gcal)&\rightarrow &H^1(k,\adgcal)&\xrightarrow{\delta} &H^2(k,\mu)\\
\downarrow&&\downarrow&&\downarrow\\
\prod_{\pfr\in V_k}H^1(k_{\pfr},\gcal)&\rightarrow&\prod_{\pfr\in V_k}H^1(k_{\pfr},\adgcal)&\rightarrow&\prod_{\pfr\in V_k}H^2(k_{\pfr},\mu)
\end{array}
\end{equation}
By \cite{BrTi, Ha}, we know that $H^1(k,\gcal)=1$ and $H^1(k_{\pfr},\gcal)=1$, for any $\pfr$. Furthermore, by Lemma \ref{l:nthroot} and Lemma \ref{l:exactR}, the last vertical map is injective. Thus following arrows in the diagram \ref{e:long}, we have that $\delta(\mathfrak{c}_1)=\delta(\mathfrak{c}_2)$ if $\phi(\mathfrak{c}_1)=\phi(\mathfrak{c}_2)$ for $\mathfrak{c}_1$ and $\mathfrak{c}_2$ in $H^1(k,\adgcal)$. So in order to show that $\phi$ is injective, it is enough to show that $\delta$ the boundary map is injective. To this end, we shall use the trick of twisting by a cocycle. Namely for a given $\mathfrak{c}$ in $H^1(k,\adgcal)$, we consider the following short exact sequence
\[
1\rightarrow \mu\rightarrow \gcal_{\mathfrak{c}} \rightarrow \adgcal_{\mathfrak{c}} \rightarrow 1.
\]
Since $\gcal_{\mathfrak{c}}$ is again simply connected $k$-group, by \cite{BrTi, Ha}, the fiber of $\delta_{\mathfrak{c}}$ over the trivial element is trivial. Hence, by the trick of twisting (see~\cite[Chapter IV, Propositon 4.3.4]{G71}), the fiber of $\delta$ over $\delta(\mathfrak{c})$ has only one element .
\end{proof}

\noindent
For a given $k$ and $l$, any admissible $\bbg$ is a $k$-inner form of $\gcal$. Hence there is an element in the image of $H^1(k,\adgcal)$ in $H^1(k,{\rm Aut}(\gcal))$ (these are Galois cohomologies and since $\adgcal$ is a smooth $k$-group scheme, the first one is naturally isomorphic to the flat cohomolgy)  which corresponds to $\bbg$. By the main inequality, Theorem~\ref{t:MainInequality}, we know that
\[
\prod_{\pfr\in T_c} q_{\pfr}\le x^{c_{12}}
\]
where $T_c$ is the set of places over which $\bbg$ is not quasi-split and $c_{12}$ just depends on $G$.
This is equivalent to saying that 
\[
\deg(D(T_c)) \le c_{12} \log(x),
\]
where $D(T_c)=\sum_{\pfr\in T_c} \pfr$. The next lemma gives us an upper bound on the number of possibilities of effective divisors, i.e. a divisor $\sum_{\pfr\in V_k} a_{\pfr} \pfr$ with non-negative coefficients, with a given upper bound on their degree.
\begin{lem}\label{l:DivDeg}
Let $k$ be any global function field; then, for any $y$,
\[
\#\{D\in \Div^+(k) |\h \deg(D)\le y\} \le 4 h_k q^y,
\]
where $\Div^+(k)$ is the set of all the effective divisors of $k$, $q=q_k$ is the number of elements of the constant field of $k$, and $h_k$ is the class number of $k$.
\end{lem}
\begin{proof}
By Riemann-Roch theorem~\cite{Ro}, there is a divisor $C$ such that for any divisor $D$,
\[
l(D)=\deg(D)-g+1+l(C-D),
\]
where $l(D)=\dim_{\f_k}\{x\in k^{\times}| (x)+D\ge 0\}\cup\{0\}$. Moreover $l(C)=g$ and $\deg(C)=2g-2$. As a corollary of Riemann-Roch theorem~\cite{Ro}, one can show that for any non-negative integer number $N$, there are $h=h_k$ effective divisors $\{D_1,\cdots, D_h\}$ of degree $N$ such that  
\[
b_N=\#\{D\in \Div^+(k)| \deg(D)=N\}=\sum_{i=1}^{h_k} \frac{q^{l(D_i)}-1}{q-1},
\]
where $q=q_k$. On the other hand, $l(C-D_i)\le l(C)=g$ as $D_i$ is an effective divisor. Hence
\[
b_N\le h\h\frac{q^{N+1}-1}{q-1}\le 2 h q^N. 
\]
Thus
\[
\#\{D\in \Div^+(k) |\h \deg(D)\le y\} \le 2h\sum_{N=0}^{y} q^N \le 4 h q^y, 
\]
which finishes the argument.
\end{proof}
\noindent
Hence by the above argument and Lemma~\ref{l:DivDeg}, for a given admissible $k$ and $l$, the number of possible sets for $T_c$ is at most $4 h_k x^{c_{12}}$. On the other hand, by Weil's theorem on Riemann hypothesis for curves over finite fields, we have
\[
h_k\le (\sqrt{q}+1)^{2g}\le q^{2g},
\]
and, by the discussions in \ref{ss:k-l}, we have that for an admissible $k$, $q^g\le x^{c_{13}}$ where $c_{13}$ just depends on $G$. Thus for a given admissible pair of function fields $(k,l)$, the number of possible $T_c$'s is at most $x^{c_{14}}$, where $c_{14}$ just depends on $G$.
\\

\noindent
Let $\mathfrak{c}$ be the element in $H^1(k,\adgcal)$ which gives us $\bbg$. By Theorem~\ref{t:adjointembedd}, it is enough to give an upper bound for the number of possible elements in $\prod_{\pfr\in V_k} H^1(k_{\pfr},\adgcal)$ for $\phi(\mathfrak{c})$. Since $\bbg$ is quasi-split over any place not in $T_c$, any admissible element in $\prod_{\pfr\in V_k} H^1(k_{\pfr},\adgcal)$ is non-trivial only in the $T_c$ components. On the other hand, $H^1(k_{\pfr},\adgcal)$ can be embedded into $H^2(k_{\pfr},\mu)$ and the latter can be embedded into a direct product of at most $2\vare(\gcal)$ torsion quotients of $\bbq/\bbz$, where each quotient has exponent at most $t=t(\gcal)$.
One gets such an embedding as the second flat cohomology of $\bbg_m$ is isomorphic to the Brauer group and the Brauer group of a non-archimedian local  field is isomorphic to $\bbq/\bbz$. Therefore, for any $\pfr$,
\[
\#H^1(k_{\pfr},\adgcal)\le c_{15},
\]
where $c_{15}$ just depends on $G$. As $\# T_c\le c_{12} \log x$, overall we have that, for a given admissible $k$ and $l$, the number of possible $\phi(\mathfrak{c})$ and therefore the number of $\bbg$ admissible $k$-forms of $\gcal$  is at most $x^{c_{16}}$, where $c_{16}$ just depends on $G$.
\subsection{Number of possible $\{P_{\pfr}\}$'s up to $\adbbg(\bba_k)$.}\label{ss:Type}
We have already given an upper bound on the number of possible $k$, $l$, and $k$-forms $\bbg$. Now, we will fix such $\bbg$, and count the number of possibilities of $\{P_{\pfr}\}$, a coherent family of parahoric subgroups. To this end, first we will give an upper bound on the number of $\Theta$ admissible types up to the action of $\adbbg(\bba_k)$, the adjoint group on local Dynkin diagrams, and then, in the next section, provide an upper bound on the class number of $\adbbg$ with respect to a coherent family of parahoric subgroups with a given admissible type.
\\

\noindent
By the main inequality, we know that 
\[
\deg(D(T_l))\le c_{17} \log_q x,
\]
where $D(T_l)=\sum_{\pfr\in T_l} \pfr$ and $c_{17}$ just depends on $G$. Hence, by Lemma~\ref{l:DivDeg}, the number of possible sets for $T_l$ is at most $4h_k x^{c_{17}}$. Thus again, by using Weil's theorem on Riemann hypothesis for curves over finite fields, we have that the number of possible sets for $T_l$ is at most $x^{c_{18}}$, where $c_{18}$ just depends on $G$. Hence, without loss of generality,  we can and we will fix $T_l$, the set of primes ramified at the local level. 
\\

\noindent
The adjoint group acts transitively on the set of hyper-special vertices. So if $\pfr$ is not in $T_l$, then either it is ramified over $l$ or $\Theta_{\pfr}$ is unique up to $\adbbg(k_{\pfr})$. Again, by the main inequality, $\# T_l\le c_{17} \log_q x$, and, on the other hand, the number of possible types for any given $\pfr$ is bounded by a constant depending only on $G$. Thus the number of possible $\Theta$'s up to $\adbbg(\bba_k)$ is at most $x^{c_{19}}$.
\subsection{Number of possible $\{P_{\pfr}\}$.}
By the discussion in~\ref{ss:Type}, we can and will fix $\{P_{\pfr}\}_{\pfr\in V^{\circ}_k}$ a coherent family of parahoric subgroups up to an element of $\adbbg(\bba_k)$. Here we will give an upper bound on the number of admissible $\{P'_{\pfr}\}_{\pfr\in V^{\circ}_k}$ within the $\adbbg(\bba_k)$-orbit  of $\{P_{\pfr}\}_{\pfr\in V^{\circ}_k}$ such that the corresponding lattices in $\bbg(k_{\pfr_0})$ are not conjugate of each other.
\\

\noindent
 Let $\overline{P}_{\pfr}$ be the stabilizer of $P_{\pfr}$ in $\adbbg(k_{\pfr})$, 
\[
\Cl(\adbbg,\{\overline{P}_{\pfr}\}_{\pfr\in V^{\circ}_k})=\adbbg(\bba_k)/\adbbg(k)\cdot \adbbg(k_{\pfr_0})\prod_{\pfr\in V^{\circ}_k} \overline{P}_{\pfr}
\]
 the class group of $\adbbg$ with respect $\{\overline{P}_{\pfr}\}_{\pfr\in V^{\circ}_k}$, and $\cl(\adbbg,\{\overline{P}_{\pfr}\})=\#\Cl(\adbbg,\{\overline{P}_{\pfr}\})$ the class number of $\adbbg$ with respect $\{\overline{P}_{\pfr}\}_{\pfr\in V^{\circ}_k}$. It is well-known that there is a correspondence between the double cosets of $\adbbg(k)$ and $\adbbg(k_{\pfr_0})\prod_{\pfr\in V^{\circ}_k} \overline{P}_{\pfr}$ in $\adbbg(\bba_k)$ and $\Cl(\adbbg,\{\overline{P}_{\pfr}\}_{\pfr\in V^{\circ}_k})$. Let $\pi$ be the projection from $\adbbg(\bba_k)$ onto 
 \[
 \adbbg(k)\setminus\adbbg(\bba_k)/ \adbbg(k_{\pfr_0})\prod_{\pfr\in V^{\circ}_k} \overline{P}_{\pfr}.
\]
  
\begin{lem}\label{l:Type}
Let $g^{(1)}, g^{(2)}\in \adbbg(\bba_k)$ such that $\pi(g^{(1)})=\pi(g^{(2)})$,
\[
\prod_{\pfr\in V^{\circ}_k}P^{(i)}_{\pfr}=g^{(i)}(\prod_{\pfr\in V^{\circ}_k}P_{\pfr}),
\]
and $\Lambda^{(i)}=\bbg(k)\cap \prod_{\pfr\in V^{\circ}_k}P^{(i)}_{\pfr}$ for $i=1,2$; then $\Lambda^{(1)}$ and $\Lambda^{(2)}$ are the same up to an element of $\adbbg(k)$.
\end{lem}
\begin{proof}
Since $\pi(g^{(1)})=\pi(g^{(2)})$, there are $g_k\in \adbbg(k)$ and $g\in \adbbg(k_{\pfr_0})\prod_{\pfr\in V^{\circ}_k} \overline{P}_{\pfr}$ such that $g^{(2)}=g_k g^{(1)} g$. By the definition, 
\[
g^{(2)}(\prod_{\pfr\in V^{\circ}_k}P_{\pfr})=g_k  g^{(1)}(\prod_{\pfr\in V^{\circ}_k}P_{\pfr}),
\]
and so $g_k(\Lambda^{(1)})=\Lambda^{(2)}$, as we claimed.
\end{proof}
\noindent
As we would like to count number of maximal lattices up to an automorphism of $G$, by the above comments and Lemma~\ref{l:Type}, it is enough to give an upper bound on $\cl(\adbbg,\{\overline{P}_{\pfr}\})$. 

\begin{thm}\label{t:class}
In the above setting, there is a constant $c$ depending only on $G$, such that
\[
\cl(\adbbg,\{\overline{P}_{\pfr}\})\le x^c.
\]
\end{thm}
\begin{proof}
By the strong approximation~\cite{Pstrong,Mstrong} for simply connected groups, we have 
\[
\bbg(\bba_k)=\bbg(k)\cdot\bbg(k_{\pfr_0})\prod_{\pfr\in V^{\circ}_k} P_{\pfr}.
\] 
Hence $\Cl(\adbbg,\{\overline{P}_{\pfr}\})$ is a quotient of 
\[
\adbbg(\bba_k)/(\adbbg(k) \Ad(\bbg(\bba_k)) \prod_{\pfr\in V_k} \overline{P}_{\pfr}).
\]
On the other hand, for any $\pfr$, we have that
\[
\adbbg(k_{\pfr})/\Ad(\bbg(k_{\pfr}))\simeq H^1(k_{\pfr},\mu).
\]
As before, let $\xi_{\pfr}$ be the homomorphism from $H^1(k_{\pfr},\mu)$ to $\Aut(\dcal_{\pfr})$. Then there is an onto map from $H^1(k_{\pfr},\mu)/\ker(\xi_{\pfr})$ to 
\[
\adbbg(k_{\pfr})/\Ad(\bbg(k_{\pfr}))\overline{P}_{\pfr}.
\]
So altogether $\Cl(\adbbg,\{\overline{P}_{\pfr}\})$ is a homomorphic image of
\[
\Ccal=(\oplus_{\pfr\in V_k} H^1(k_{\pfr},\mu)/\ker(\xi_{\pfr}))/\Delta(H^1(k,\mu)),
\]
where $\Delta$ is the natural diagonal homomorphism. 
\\

\noindent
First, we assume that $\bbg$ is an inner form. In this case, $\mu$ is isomorphic to $(\mu_t)^{\vare}$ as a $k$-group scheme, where $t=t(\gcal)$ and $\vare=\vare(\gcal)$. Hence $H^1(k_{\pfr},\mu)\simeq (k_{\pfr}^{\times}/{k_{\pfr}^{\times}}^{t})^{\vare}$ and  $H^1(k,\mu)\simeq (k^{\times}/{k^{\times}}^{t})^{\vare}$. If $\pfr$ is not in $T_c$, i.e. $\bbg$ is split over $k_{\pfr}$, then, by \cite[Lemma 2.3, Proposition 2.7]{P}, $\ker{\xi_{\pfr}}=({k_{\pfr}^{\times}}^{t}\o_{\pfr}^{\times}/{k_{\pfr}^{\times}}^{t})^{\vare}$. Hence a subgroup of index at most $c_{15}^{\# T_c}$ in $\Ccal$ is a homomorphic image of
\[
((\oplus_{\pfr\in V_k} k_{\pfr}^{\times}/ ({k_{\pfr}^{\times}}^{t}\o_{\pfr}^{\times}))/\Delta(k^{\times}/{k^{\times}}^t)^{\vare} \simeq (\Div(k)/(t \Div(k)+(k))))^{\vare},
\]
where $\Div(k)$ is the group of divisors of $k$ and $(k)$ is its subgroup of principal divisors. Since the latter has at most $h_k\le x^{c_{20}}$ elements and $\# T_c\le c_{12} \log x$, $\cl(\adbbg,\{\overline{P}_{\pfr}\})\le \# \Ccal$ is at most $x^{c_{21}}$, where all the constants just depend on $G$, which is the desired result.
\\

\noindent
Now, let us assume that $\bbg$ is an outer form. In this case, $\mu$ is isomorphic  to $\nu^{(1)}=R^{(1)}_{l/k}(\mu_n)$ as  a $k$-group scheme, where $R^{(1)}_{l/k}(\mu_n)$ is as in Lemma~\ref{l:exactR} and $n=t(\bbg)$ unless $\bbg$ is of type $D_r$ with $r$ even. Hence we have the following exact sequence
\[
\mu_n(l\otimes_k k_{\pfr})\xrightarrow{N_{l/k}}\mu_n(k_{\pfr})\rightarrow H^1(k_{\pfr},\mu) \rightarrow (l\otimes_k k_{\pfr})^{\times}/{(l\otimes_k k_{\pfr})^{\times}}^n\xrightarrow{N_{l/k}} k_{\pfr}^{\times}/{k_{\pfr}^{\times}}^n,
\]
for any $\pfr$ and a similar exact sequence for $k$ instead of $k_{\pfr}$. For any $\pfr$, $l\otimes_k k_{\pfr}\simeq \oplus_{\pfrl |\pfr} l_{\pfrl}$ and moreover again by \cite[Lemma 2.3, Proposition 2.7]{P}, if $\pfr$ is not ramified over $l$ and not in $T_c$, i.e. $\bbg$ is quasi-split over $k_{\pfr}$ and splits over $\widehat{k}_{\pfr}$, then the image of $\ker(\xi_{\pfr})$ contains the intersection of the kernel of the norm map and $\oplus_{\pfrl |\pfr} \o_{\pfrl}^{\times}{l_{\pfrl}^{\times}}^n/{l_{\pfrl}^{\times}}^n$. On the other hand, the following diagram is commutative
\[
\begin{array}{ccc}
\mu_n(k_{\pfr}^{\times})/N_{l/k}(\mu_n(l\otimes_k k_{\pfr}))& \rightarrow & H^1(k_{\pfr},\mu)\\
\downarrow&&\downarrow\\
\mu_n(\widehat{k}_{\pfr}^{\times})/N_{l/k}(\mu_n(l\otimes_k \widehat{k}_{\pfr}))& \rightarrow & H^1(\widehat{k}_{\pfr},\mu),
\end{array}
\]
and so if $\pfr$ is not ramified over $l$ and $\bbg$ is quasi-split over $\pfr$, then the image of $\mu_n(k_{\pfr}^{\times})/N_{l/k}(\mu_n(l\otimes_k k_{\pfr}))$ is in $\ker(\xi_{\pfr})$. 
\\

\noindent Form the given long exact sequence, one can get an exact sequence for $\Ccal$ and by the above discussions both the kernel and the cokernel have at most $x^{c_{22}}$ elements, where $c_{22}$ just depends on $G$, and so we get the desired result.
\end{proof}

\noindent
Theorem~\ref{t:class} completes the proof of the upper bound part of Theorem~\ref{t:maximal}. 
Indeed by Remark~\ref{r:dJK}, we also see that an affirmative answer to Question~\ref{q:Maximal} gives us a polynomial upper bound on the number of possible maximal lattices.
\subsection{A lower bound on the number of maximal lattices.}
By Tits' classification, there is $k$ a global function field, a place $\pfr_0$, and $\bbg$ a simply connected absolutely almost simple $k$-group, such that
\begin{itemize}
\item[1-] $k_{\pfr_0}\simeq K$.
\item[2-] $\bbg\simeq\bbg_0$ as $K$-groups after identifying $k_{\pfr_0}$ with $K$.
\end{itemize}
Let $\{P^m_{\pfr}\}_{\pfr\in V_k^{\circ}}$ be a coherent family of parahoric subgroups of maximum volume and $\{P'_{\pfr}\}_{\pfr\in V_k^{\circ}}$ be a family of parahoric subgroups of $\{\Theta'_{\pfr}\}_{\pfr\in V_k^{\circ}}$ such that
\begin{itemize}
\item[1-] If either $\bbg$ is not quasi-split over $k_{\pfr}$ or it does not split over $\widehat{k}_{\pfr}$, then $P'_{\pfr}=P^m_{\pfr}$. 
\item[2-] If $\bbg$ is not an inner form of type A, then $\Theta'_{\pfr}$ is a single vertex which is not hyper-special whenever $\bbg$ is quasi-split over $k_{\pfr}$ and splits over $\widehat{k}_{\pfr}$.
\item[3-] If $\bbg$ is an inner form of type $\A_r$, then we can and will assume that $k=\f_q(t)$ and $|f/g|_{\pfr_0}=q^{\deg(f)-\deg(g)}$. In this case,  whenever $\bbg$ splits over $k_{\pfr}$, its local Dynkin diagram over $k_{\pfr}$ is a cycle of length $r+1$ and $\adbbg(k_{\pfr})$ acts on it by rotations. Let $p$ be a prime factor of $r+1$ and $\Theta'_{\pfr}$ an orbit of rotation of length $(r+1)/p$. In particular, $\#\Theta'_{\pfr}=p$. 
\end{itemize}
For any $D$ a finite subset of $V^{\circ}_k$, let  
\[
P^D_{\pfr}=
\begin{cases}
P'_{\pfr}& \text{if $\pfr\in D$}\\
P_{\pfr}^m& \text{if $\pfr\in V_k^{\circ}\setminus D$}
\end{cases}
,\hspace{1cm}\Lambda_D=\bbg(k)\cap\prod_{\pfr\in V_k^{\circ}} P^D_{\pfr},
\]
and $\Gamma_D=N_G(\Lambda_D)$.
\begin{lem}\label{l:MaximalCondition}
In the above setting, $\Gamma_D$ is a maximal lattice in $G$ for any $D$ and moreover $\Gamma_D \cap \bbg(k)=\Lambda_D$.
\end{lem}
\begin{proof}
Let $\Gamma$ be a maximal lattice which contains $\Gamma_D$ and $\Lambda=\Gamma\cap \bbg(k)$. Then by Rohlfs' maximality criterion (for the treatment in the case of positive characteristic see~\cite[Section2]{BP}), there is $\{P_{\pfr}\}$ a coherent family of parahoric subgroups such that
\[
\Lambda=\bbg(k)\cap \prod_{\pfr\in V_k^{\circ}} P_{\pfr},
\]
and $\Gamma=N_G(\Lambda)$. Thus $\Lambda_D\subseteq \Lambda$ and so $P^D_{\pfr}\subseteq P_{\pfr}$, for any $\pfr\in V_k^{\circ}$. If $\bbg$ is not an inner form of type A, then $P^D_{\pfr}$ is a maximal parahoric subgroup, and so we are done. So let us assume that $\bbg$ is an inner form of type $\A_r$ and $P^D_{\pfr}$ is a proper parahoric subgroup of $P_{\pfr}$, for some $\pfr$. Because of the way we defined $P^D_{\pfr}$, $\bbg$ splits over $k_{\pfr}$ and no element of $\Aut(\dcal_{\pfr})$ preserves type of $P_{\pfr}$ whenever $P^D_{\pfr}\neq P_{\pfr}$. Now we will appeal to Rohlfs' exact sequences (for a treatment of the positive characteristic case see~\cite[Proposition 2.9]{BP}).
\[
1\rightarrow \mu(k_{\pfr_0})/\mu(k) \rightarrow \Gamma/\Lambda \xrightarrow{\delta} \delta(\adbbg(k))_{\Theta^{\circ}}\rightarrow 1,  
\]
and a similar short exact sequence for $\Gamma_D/\Lambda_D$. Since $\Gamma_D$ is a subgroup of $\Gamma$, $\delta(\adbbg(k))_{\Theta^D}$ is a subgroup of $\delta(\adbbg(k))_{\Theta^{\circ}}$, where $\Theta^D=\{\Theta_{\pfr}^D\}_{\pfr\in V_k^{\circ}}$ is the set of types of the parahoric subgroups $\{P^D_{\pfr}\}$ and $\Theta^{\circ}=\{\Theta_{\pfr}\}_{\pfr\in V_k^{\circ}}$ is the set of types of the parahoric subgroups $\{P_{\pfr}\}$. On the other hand, by a result of Harder~\cite{Ha}, $H^1(k,\bbg)$ is trivial and so $\delta(\bbg(k))=H^1(k,\mu)$. Since $\bbg$ is an inner form of type $\A_r$, $\mu$ is isomorphic to $\mu_{r+1}$ as a $k$-group scheme. Thus $H^1(k,\mu)$ is isomorphic to $k^{\times}/(k^{\times})^{r+1}$. If $P^D_{\pfr}$ is a proper subgroup of $P_{\pfr}$, then $\pfr^{(r+1)/p}(k^{\times})^{r+1}$ is in $\delta(\adbbg(k))_{\Theta^D}$ but not in $\delta(\adbbg(k))_{\Theta^{\circ}}$, which is a contradiction and completes the proof. 

\end{proof}
\begin{lem}\label{l:LatticeDiffer}
In the above setting, there is a fixed finite set $D_0$ of places of $k$ such that if $\theta(\Gamma_{D_1})=\Gamma_{D_2}$ where $D_1$ and $D_2$ are two finite subsets of $V_k^{\circ}$ and $\theta\in \Aut(G)$, then there is $\theta_1\in \Aut(k)$ such that the symmetric difference of $D_1$ and  $\theta_1(D_2)$
\[
D_1 \bigtriangleup \theta_1(D_2)=(D_1\setminus \theta_1(D_2)) \cup (\theta_1(D_2)\setminus D_1)
\]
 is a subset of $D_0$.
\end{lem} 
\begin{proof}
First we will prove that $\theta(\bbg(k))=\bbg(k)$.
Since $\Lambda_{D_i}\subseteq \bbg(k)$'s are of finite index in $\Gamma_{D_i}$'s, going to $\Lambda_1\subseteq \bbg(k)$ a finite index subgroup of $\Lambda_{D_1}$ we can assume that $\theta(\Lambda_1)=\Lambda_2\subseteq \bbg(k)$. By a theorem of Margulis~\cite[Theorem C, Chapter VIII]{Mar}, there exist an automorphism $\theta_1:k\rightarrow k$, a $k$-isomorphism $\theta_2: \leftexp{\theta_1}{\bbg}\rightarrow \bbg$ and a homomorphism $\theta_3:\Lambda_1\rightarrow \mu(k_{\pfr_0})$ such that
\[
\theta(\lambda)=\theta_3(\lambda)\cdot\theta_2(\theta_1(\lambda)),
\]
for all $\lambda\in \Lambda_1$. As $\mu(k_{\pfr_0})$ is a finite group, on a lattice $\theta$ is equal to $\theta_2\circ\theta_1$. Hence by a theorem of G.~Prasad~\cite{Pra}, $\theta=\theta_2\circ\theta_1$. In particular, $\theta(\bbg(k))=\bbg(k)$ and it can be uniquely extended to a continuous isomorphism between $\bbg(k_{\pfr})$ and $\bbg(k_{\theta_1(\pfr)})$ for any $\pfr$. \\

\noindent
By Lemma~\ref{l:MaximalCondition} and the above discussion, $\theta(\Lambda_{D_1})=\Lambda_{D_2}$. Therefore, again by the above discussion, $\theta(P^{D_1}_{\pfr})=P^{D_2}_{\theta_1(\pfr)}$, for any $\pfr\in V_k^{\circ}$. However, if $\bbg$ is quasi-split over $k_{\pfr}$, $\bbg$ splits over $\widehat{k}_{\pfr}$, and $\pfr\in D_1\bigtriangleup \theta_1(D_2)$, then on one hand $\Theta^{D_1}_{\pfr}$ and $\Theta^{D_2}_{\theta_1(\pfr)}$ can be considered to be the same subset of $\dcal_{\pfr}$, after identifying $\dcal_{\pfr}$ with $\dcal_{\theta_1(\pfr)}$, and on the other hand one of them is hyper-special and the other one is not, which is a contradiction. Therefore $D_1\bigtriangleup \theta_1(D_2)$ is a subset of $D_0$ the set of places such that either $\bbg$ is not quasi-split over $k_{\pfr}$ or it does not split over $\widehat{k}_{\pfr}$, which completes our proof.
\end{proof}
\noindent
We notice that, by Theorem \ref{t:covolume},
\[
\vol(G/\Gamma_D)\le c_{23}\h q_k^{\dim\bbg\cdot\deg(\Div(D))},
\]
where $\Div(D)=\sum_{\pfr\in D} \pfr$ and $c_{23}$ just depends on $\bbg$ and $k$. On the other hand, as a consequence of Weil's theorem on Riemann hypothesis for curves over finite fields~\cite[Proposition 17.2]{Ro}, for any $N$, the number of square-free effective divisors of $k$ whose degree is $N$ is at least $c_{24} q_k^N,$ where $c_{24}$ only depends on $k$. We also know that the group of automorphism of $k$ is finite~\cite{Sc}. Hence, by Lemmas~\ref{l:MaximalCondition} and \ref{l:LatticeDiffer}, we get the desired polynomial lower bound on the number of maximal lattices in $G$.
\subsection{An upper bound on $\#\Gamma/\Lambda$.}
Here we will give a polynomial upper bound on $\#\Gamma/\Lambda$ which is needed on counting the number of all the lattices with covolume at most $x$.\\

\noindent
Let $\Gamma$ be a maximal lattice in $G$ with covolume at most $x$ and $k$, $l$, $\bbg$, $\Lambda$, and $\{P_{\pfr}\}_{\pfr\in V_k^{\circ}}$ as before. 
\begin{lem}\label{l:gamma/lambda}
In the above setting, $\#\Gamma/\Lambda\le x^c$, where $c$ just depends on $G$. 
\end{lem}
\begin{proof} By Rohlfs' short exact sequence~\cite[Proposition 2.9]{BP} and arguments in~\ref{ss:LowerBound}, we know that
\[
\#\Gamma/\Lambda\le c_{25} \#H^1(k,\mu)_{\xi}\cdot\prod_{\pfr\in V_k} \#\Xi_{\pfr}\cdot \#{\rm Aut}\h \dcal_{\pfr_0},
\]
where $c_{25}$ just depends on $G$.  By the main inequality, we have $\#\mathfrak{R}(\Gamma)\le \log x$. On the other hand, $\#\Xi_{\pfr}\le c_{26}$ where $c_{26}$ just depends on $G$. Now, by Proposition~\ref{p:H1Upper}, the upper  bounds obtained for $g_k$ and $g_l$ in the proof of the main inequality and finally using Weil's Riemann hypothesis,  we can conclude the desired inequality. 
\end{proof}
\section{Counting all the lattices and subgroup growth.}
Sections~\ref{ss:ReductionProP} and \ref{ss:ReductionLieAlgebra} are devoted to the proof of the upper bound of Theorem~\ref{t:CountingLattices} modulo Theorem~\ref{t:GradedLie}, which is proved in Section \ref{ss:GradedLie}, the lower bound of Theorem~\ref{t:CountingLattices} is given in Section~\ref{ss:LatticesLowerBound} and, finally, the proof of Theorem~\ref{t:SubgroupGrowth} is completed in Section~\ref{ss:SubgroupGrowth}.  
\subsection{Reduction to subgroup growth of certain pro-$p$ groups.}~\label{ss:ReductionProP}
Let $G=\bbg_0(K)$ as above. For any $x$, let $\mathfrak{M}_x$ be a set of representatives of maximal lattices in $G$ with covolume at most $x$ up to $\Aut(G)$. Hence, by Theorem~\ref{t:maximal}, 
\begin{align}
\notag \rho_x(G)\le& \sum_{\Gamma\in \mathfrak{M}_x} s_{x/\vol(G/\Gamma)}({\Gamma})\le \#\mathfrak{M}_x\cdot \max_{\Gamma\in \mathfrak{M}_x} s_{x}({\Gamma})\\
\notag \le &x^{B \log\log x}\cdot  \max_{\Gamma\in \mathfrak{M}_x} s_{x}({\Gamma}).
\end{align}
Thus, in order to prove the upper bound of Theorem~\ref{t:CountingLattices}, it is enough to show the following.
\begin{thm}\label{t:MaxSubgroupGrowthUpperBound}
Let $\Gamma$ be a maximal lattice in $G$. If CSP, MP and Weil conjecture hold and $\vol(G/\Gamma)\ll x$, then 
\[
\log s_x(\Gamma)\ll (\log x)^2,
\]
 where the implied constants only depend on $G$.
\end{thm}

\noindent
In this section we reduce the proof of Theorem \ref{t:MaxSubgroupGrowthUpperBound} to 
understanding the subgroup growth of certain pro-$p$ groups.  Let $\Gamma$ be a maximal 
lattice in $G$ with covolume at most $x$ and $k$, $l$, $\pfr_0$, $\bbg$, $\Lambda$ and 
$\pcal=\{P_{\pfr}\}_{\pfr\in V_k^{\circ}}$ as before. Let $\mathfrak{R}(\pcal)$ be the set of all 
the places in $V_k^{\circ}$ such that $P_{\pfr}$ is not hyper-special. 
\begin{lem}\label{l:Recall}
In the above setting, we have that
\begin{itemize}
\item[1-] $\deg(\Div(\mathfrak{R}(\pcal)))\le c_2 \log x$ where $\Div(\mathfrak{R}(\pcal))=\sum_{\pfr\in \mathfrak{R}(\pcal)}\pfr$ and $c_2$ just depends on $G$. 
\item[2-] $\#\Gamma/\Lambda\le x^{c_3}$, where $c_3$ just depends on $G$.
\end{itemize}
\end{lem}
\begin{proof}
Note that $\deg(\Div(\mathfrak{R}(\pcal)))\le \deg(\Div(\mathfrak{R}(\Gamma)))+\sum_{\pfr/l\h ramified} \deg(\pfr)$. By Hurwitz genus formula, $\sum_{\pfr/l\h ramified} \deg(\pfr)\le 2[\f_l:\f_k](g_l-1)+2[l:k]$, and, by the main inequality, $\deg(\mathfrak{R}(\Gamma))\le c_4 \log x$ where $c_4$ just depends on $G$. On the other hand, by discussions in  \ref{ss:k-l}, $g_l\le c_5 \log x$ where $c_5$ just depends on $G$. This finishes the proof of the first part. The second part is the direct consequence of Lemma~\ref{l:gamma/lambda}.
\end{proof}
\noindent
By Lemma~\ref{l:Recall} and ~\cite[Lemma 1.2.2, Proposition 1.3.2]{LS},  
\[
s_x({\Gamma})\le s_x(\Lambda)\cdot s_x({\Gamma/\Lambda})\cdot x^{c_3 \log x}\le s_x({\Lambda})\cdot x^{(c^2_3+c_3)\log x}.
\]
So in order to get the desired result it is enough to give the right upper bound for $s_x({\Lambda})$.
Let $\widehat{\Lambda}$ be the profinite closure of $\Lambda$. Then there is a correspondence between the open subgroups of $\widehat{\Lambda}$ and subgroups of finite index in $\Lambda$. 
As we assume the congruence subgroup property and Margulis-Platonov's conjecture, because of strong approximation, $\widehat{\Lambda}\simeq \prod_{\pfr\in V_k^{\circ}}P_{\pfr}$.\\

\noindent
Using a result of Babai-Cameron-Palfy~\cite{BCP}, Lubotzky~\cite[Proposition 4.3]{Lu} proved that if $H$ is a subgroup of index at most $x$ in $\prod_{\pfr\in V_k^{\circ}}P_{\pfr}$, then it contains a subnormal subgroup of $\prod_{\pfr\in V_k^{\circ}}P_{\pfr}$ of index at most $x^{c_6}$ where $c_6$ just depends on $G$.\footnote{In~\cite{Lu}, $\bbg$ is assumed to be split. However the same argument works without this assumption.} Hence, by~\cite[Lemma 1.2.3]{LS},
\[
s_x({\prod_{\pfr\in V_k^{\circ}}P_{\pfr}})\le x^{c_6^2 \log x}\cdot s_{x^{c_6}}^{\triangleleft\triangleleft}(\prod_{\pfr\in V_k^{\circ}}P_{\pfr}),
\]
where $s_x^{\triangleleft\triangleleft}(\bullet)$ is the number of subnormal subgroups of index at most $x$. Thus it is enough to get the right upper bound for $s_x^{\triangleleft\triangleleft}(\prod_{\pfr\in V_k^{\circ}}P_{\pfr}).$ 
\begin{lem}\label{l:Subnormal}
Let $H$ be a subnormal subgroup of $\prod_{\pfr\in V_k^{\circ}}P_{\pfr}$ of index at most $x$. Then there is $V(H)$ a finite subset $V_k$ such that
\begin{itemize}
\item[1-] If $\pfr$ is not in $V(H)$, then $P_{\pfr}\subseteq H$.
\item[2-] $\deg(\Div(V(H)))\le c_7 \log x$, where $\Div(V(H))=\sum_{\pfr\in V(H)}\pfr$ and $c_7$ just depends on $G$.
\end{itemize}
\end{lem}
\begin{proof}
Let $V(H)=\{\pfr\in V_k^{\circ}|\h H\cap P_{\pfr}\neq P_{\pfr}\}$. By Lemma~\ref{l:Recall}, $\deg(\mathfrak{R}(\pcal)) \le c_2 \log x$. So it is enough to focus on places where $P_{\pfr}$ is hyper-special. Let $P^{(1)}_{\pfr}$ be the first congruence subgroup of $P_{\pfr}$. As $P_{\pfr}$ is a hyper-special parahoric subgroup, $P_{\pfr}/P^{(1)}_{\pfr}$ is a finite quasi-simple group. Hence any proper normal subgroup of $P_{\pfr}/P^{(1)}_{\pfr}$ is contained in its center. Let $Z_{\pfr}$ be the preimage of the center of $P_{\pfr}/P^{(1)}_{\pfr}$. By the above discussion, if $\pfr\in V(H)\setminus \mathfrak{R}(\pcal)$, then either $H\cap P_{\pfr}$ is contained in $Z_{\pfr}$ or $(H\cap P_{\pfr})P^{(1)}_{\pfr}=P_{\pfr}$. By~\cite[Lemma 4.7]{Lu}, if $(H\cap P_{\pfr})P^{(1)}_{\pfr}=P_{\pfr}$, then $H\cap P_{\pfr}=P_{\pfr}$ as $H\cap P_{\pfr}$ is a subnormal subgroup of $P_{\pfr}$. We claim that if $\pfr \in V(H)\setminus \mathfrak{R}(\pcal)$, then $\pi_{\pfr}(H)\subseteq Z_{\pfr}$. By the same argument as above, $\pi_{\pfr}(H)\subseteq P_{\pfr}$ if $\pi_{\pfr}(H)\nsubseteq Z_{\pfr}$. Let
$
H=N_s\lhd N_{s-1} \lhd \cdots \lhd N_1\lhd \prod_{\pfr\in V_k^{\circ}}P_{\pfr}, 
$
and assume that $P_{\pfr}\subseteq N_i$ and $P_{\pfr}\nsubseteq N_{i+1}$. So for any $a\in P_{\pfr}$, there is $y_a\in \prod_{\pfr'\in V_k^{\circ}\setminus\{\pfr\}} P_{\pfr'}$ such that $(a,y_a)$ is in $N_{i+1}$. On the other hand, $N_{i+1}$ is a normal subgroup of $N_i$ and $P_{\pfr}$ is contained in $N_i$. Hence, for any $a,a'\in P_{\pfr}$, $(a,y_a)^{-1}(a',1)^{-1}(a,y_a)(a',1)=(a^{-1}a'^{-1}aa',1)$ is in $N_{i+1}$. Since $P_{\pfr}$ is perfect, it is a subgroup of $N_{i+1}$, which is a contradiction.  
\\

\noindent
Overall, we have shown that 
\[
H\subseteq \prod_{\pfr\in\mathfrak{R}(\pcal)}P_{\pfr}\cdot\prod_{\pfr\in V(H)\setminus\mathfrak{R}(\pcal)} Z_{\pfr}\cdot \prod_{\pfr\not\in V(H)\cup \mathfrak{R}(\pcal)\cup\{\pfr_0\}} P_{\pfr}.
\]
In particular, the index of $H$ in $ \prod_{\pfr\in V_k^{\circ}}P_{\pfr}$ is at least a constant power of $\prod_{\pfr\in V(H)\setminus  \mathfrak{R}(\pcal)} q_{\pfr}$ and so combining Lemma~\ref{l:Recall}, we have the desired result.
\end{proof}
\noindent
By Lemma~\ref{l:Subnormal}, we have that
\[
s^{\triangleleft\triangleleft}_x(\prod_{\pfr\in V_k^{\circ}}P_{\pfr})\le \sum_{T\in \mathcal{V}_k(c_7 \log x)}  s^{\triangleleft\triangleleft}_x(\prod_{\pfr\in T}P_{\pfr})\le \# \mathcal{V}_k(c_7 \log x)\cdot \sup_{T\in \mathcal{V}_k(c_7 \log x)} s_x(\prod_{\pfr\in T}P_{\pfr}),
\]
where $\mathcal{V}_k(y):=\{T\subseteq V_k^{\circ}| \sum_{\pfr\in T}\deg(\pfr)\le y\}$. By Lemma~\ref{l:DivDeg} and the argument after that, $\# \mathcal{V}_k(c_7 \log x)\le x^{c_8}$, where $c_8$ just depends on $G$, and so it is enough to get the desired upper bound only for $s_x(\prod_{\pfr\in T}P_{\pfr})$, where $T\in \mathcal{V}_k(c_7 \log x)$. We note that $P^{(1)}_{\pfr}$ the first congruence subgroup of $P_{\pfr}$ is a pro-$p$ normal subgroup of $P_{\pfr}$ and the index of $\prod_{\pfr\in T}P^{(1)}_{\pfr}$ in $\prod_{\pfr\in T}P_{\pfr}$ is at most $x^{c_9}$, where $c_9$ only depends on $G$. Hence, again by~\cite[Lemma 1.2.2, Proposition 1.3.2]{LS},
\[
s_x(\prod_{\pfr\in T}P_{\pfr})\le s_x(\prod_{\pfr\in T}P^{(1)}_{\pfr})\cdot s_x(\prod_{\pfr\in T}P_{\pfr}/P^{(1)}_{\pfr})\cdot x^{c_9\log x}\le s_x(\prod_{\pfr\in T}P^{(1)}_{\pfr})\cdot  x^{(c_9^2+c_9)\log x}.
\]
Thus the right upper bound for $s_x(\prod_{\pfr\in T}P^{(1)}_{\pfr})$, where  $T\in \mathcal{V}_k(c_7 \log x)$, finishes the proof of Theorem~\ref{t:MaxSubgroupGrowthUpperBound}.
\subsection{Reduction to Theorem~\ref{t:GradedLie} on graded Lie algebras.}~\label{ss:ReductionLieAlgebra}
Following~\cite{LSa,ANS}, in order to understand $s_x(\prod_{\pfr\in T}P^{(1)}_{\pfr})$, where  $T\in \mathcal{V}_k(c_7 \log x)$, we will work with the associated graded Lie algebra. In this section, we complete the proof of Theorem~\ref{t:CountingLattices} modulo Theorem~\ref{t:GradedLie}, which will be proved in \ref{ss:GradedLie}.
\\

\noindent
Let $\lcal$ be the graded Lie algebra associated with the filtration $\prod_{\pfr\in T} P^{(i)}_{\pfr}$ (for the definition of $P^{(i)}_{\pfr}$ and the associated graded Lie algebras see~Sections \ref{ss:Filtration} and \ref{ss:AssociatedGraded}.). Thus $\lcal\simeq \oplus_{\pfr\in T} L_{\Theta_{\pfr}}$.
\begin{prop}\label{p:GradedLie}
In the above setting, let $\cal$ be an $\f_p$-subalgebra of $\lcal$ with finite $\f_p$-codimension; then there are positive numbers $c_{10}$ and $c_{11}$ depending only on $G$ such that
\[
\cdim_{\cal}[\cal,\cal]\le c_{10}\deg T+c_{11}\h\cdim_{\lcal} \cal,
\] 
where $\deg T=\sum_{\pfr\in T} \deg \pfr$.
\end{prop}
\begin{proof}
Let $\F$ be the algebraic closure of $\f_p$ and $\Hfr=\cal\otimes_{\f_p} \F$; then $\Hfr$ is an $\F$-subalgebra of
\[
\lcal\otimes_{\f_p}\F\simeq (\lcal\otimes_{\f_p}\f_k)\otimes_{\f_k}\F\simeq (\oplus_{\pfr\in T}\L^{\deg \pfr}_{\Theta_{\pfr}})^{[\f_k:\f_p]},
\]
and the $\f_p$-codimension of $\cal$ in $\lcal$ is equal to the $\F$-codimension of $\Hfr$ in $\lcal\otimes_{\f_p}\F$.  Note that the Galois group of $\F/\f_p$ acts on  $\lcal\otimes_{\f_p} \F$ and $\lcal$ is the subalgebra which is fixed by this action. We claim that $\cdim_{\cal} [\cal,\cal]=\cdim_{\Hfr} [\Hfr,\Hfr]$. Let $\{h_i\}$ be an $\f_p$-basis of $\cal$; then the $\f_p$-span (resp. $\F$-span) of $\{[h_i,h_j]\}$ is equal to $[\cal,\cal]$ (resp. $[\Hfr,\Hfr]$). Hence $[\cal,\cal]$ is the $\f_p$-structure of $[\Hfr,\Hfr]$ under the above Galois group action. Thus $\cal/[\cal,\cal]$ is an $f_p$-structure of $\Hfr/[\Hfr,\Hfr]$, which gives the claimed equality.
\\

\noindent
On the other hand, by Corollary~\ref{c:ChangeParahorics},  $(\oplus_{\pfr\in T}\L^{\deg \pfr}_{\Theta_{\pfr}})^{[\f_k:\f_p]},
$ has an $\F$-subalgebra of $\F$-codimension at most $c_{12}\deg T$  ($c_{12}$ just depends on $G$) which is an $\F$-subalgebra of $\F$-codimension at most $c_{12}\deg T$ in $\L=\oplus_{\pfr\in T}\L^{[\f_k:\f_p]\deg \pfr}_{\pfr,\psi_s}$ where $\L_{\pfr,\psi_s}$ is the graded Lie algebra associated with the filtration of the parahoric subgroup of the  type $\{\psi_s\}$ in $\bbg(\widehat{k}_{\pfr})$ (for the definition of $\psi_s$ see~\ref{ss:IwahoriBasis}.). Hence one can find $\hfr$ an $\F$-subalgebra of $\Hfr$ such that
\begin{itemize}
\item[1-] $\cdim_{\Hfr}\hfr \le c_{12} \deg T$.
\item[2-] $\hfr$ can be embedded in $\L$ as an $\F$-subalgebra, and
\begin{equation}\label{e:LastCodim}
\cdim_{\L} \hfr\le \cdim_{\lcal\otimes_{f_p}\F} \Hfr+  c_{12} \deg T=\cdim_{\lcal} \cal +  c_{12} \deg T.
\end{equation}
\end{itemize}
\noindent
By the above discussion and inequality, we have that 
\begin{equation}\label{e:FirstCodim}
\cdim_{\cal}[\cal,\cal]=\cdim_{\Hfr}[\Hfr,\Hfr]\le \cdim_{\Hfr}[\hfr,\hfr]\le \cdim_{\hfr}[\hfr,\hfr]+c_{12} \deg T.
\end{equation}

\noindent
By Corollary~\ref{c:SpecialAlgebra}, we know that $\L_{\pfr,\psi_s}\simeq\oplus^{\infty}_{i=1}\gfr_{\pfr,i ({\rm mod}\h r_{\phi})}\otimes t^i$, where $\gfr_{\pfr}=\oplus^{r_{\phi}-1}_{i=1}\gfr_{\pfr,i}$ is a finite dimensional  perfect $\F$-algebra. Moreover $\gfr_{\pfr}$ just depends on the type of $\bbg$ over $\widehat{k}_{\pfr}$. Thus
\[
\L\simeq \L(\gfr_{{\rm s}})^{[\f_k:\f_p]\deg T_{{\rm s}}}\oplus\L(\gfr_{{\rm r}})^{[\f_k:\f_p]\deg T_{{\rm r}}}\oplus\L(\gfr_{{\rm n}})^{[\f_k:\f_p]\deg T_{{\rm n}}}
\]
where $T_{{\rm s}}$ consists of the places in $T$ over which $\bbg$ splits, $T_{{\rm r}}$ consists of the places in $T$ over which $\bbg$ does not split but $\Phi$ is reduced, and $T_{{\rm n}}$ consists of the places in $T$ over which $\Phi$ is non-reduced, $\L(\gfr_{\bullet})$'s are the associated graded Lie algebras, and $\deg T_{\bullet}=\sum_{\pfr\in T_{\bullet}} \deg \pfr$. 
\\

\noindent
Let $\L_{\rm s}=\L(\gfr_{{\rm s}})^{[\f_k:\f_p]\deg T_{{\rm s}}}$ and $\hfr_{\rm s}=\hfr\cap  \L_{\rm s}$; similarly define $\L_{\rm r}$, $\L_{\rm n}$, $\hfr_{\rm r}$ and $\hfr_{\rm n}$. Therefore, by Theorem~\ref{t:GradedLie}, there is $c_{13}$ a constant depending on $G$, such that
\begin{itemize}
\item[1-] $\cdim_{ \L_{\rm s}} [\hfr_{\rm s},\hfr_{\rm s}]\le c_{13} (\cdim_{ \L_{\rm s}} \hfr_{\rm s}+ [\f_k:\f_p]\deg T_{\rm s}),$
\item[2-] $\cdim_{ \L_{\rm r}} [\hfr_{\rm r},\hfr_{\rm r}]\le c_{13} (\cdim_{ \L_{\rm r}} \hfr_{\rm r}+ [\f_k:\f_p]\deg T_{\rm r}),$
\item[3-] $\cdim_{ \L_{\rm n}} [\hfr_{\rm n},\hfr_{\rm n}]\le c_{13} (\cdim_{ \L_{\rm n}} \hfr_{\rm n}+ [\f_k:\f_p]\deg T_{\rm n}).$
\end{itemize}
Hence we have
\begin{equation}\label{e:MainCodim}
\begin{array}{ll}
\cdim_{\hfr}[\hfr,\hfr]&\le \cdim_{ \L_{\rm s}} [\hfr_{\rm s},\hfr_{\rm s}]+ \cdim_{ \L_{\rm r}} [\hfr_{\rm r},\hfr_{\rm r}]+\cdim_{ \L_{\rm n}} [\hfr_{\rm n},\hfr_{\rm n}]+ 3 \cdim_{\L}\hfr\\
&\\
&\le c_{13}(3\h\cdim_{\L}\hfr+[\f_k:\f_p]\deg T)+ 3\h\cdim_{\L}\hfr\\
&\\
&=c_{13}[\f_k:\f_p]\deg T+ 3(c_{13}+1) \cdim_{\L}\hfr.
\end{array}
\end{equation}
Now inequalities~\ref{e:LastCodim},~\ref{e:FirstCodim}, and~\ref{e:MainCodim} complete our proof.
\end{proof}
\begin{proof}[Proof of Theorem~\ref{t:MaxSubgroupGrowthUpperBound} modulo Theorem~\ref{t:GradedLie}]
By the previous discussions, we only have to give the right upper bound for $s_x(\prod_{\pfr\in T} P^{(1)}_{\pfr})$, where $T\in \mathcal{V}_k(c_7 \log x)$. We complete the proof as in~\cite[Page 115]{LS}. Let $H$ be a subgroup of index $p^d$ in $\prod_{\pfr\in T} P^{(1)}_{\pfr}$ and 
\[
\lcal(H)=\oplus^{\infty}_{i=1} (H\cap \prod_{\pfr\in T} P^{(i)}_{\pfr}) \prod_{\pfr\in T} P^{(i+1)}_{\pfr} /\prod_{\pfr\in T} P^{(i+1)}_{\pfr},
\]
the associated $\f_p$-subalgebra of $\lcal$. It is clear that the $\f_p$-codimension of $\lcal(H)$ in $\lcal$ is $d$ and $[\lcal(H),\lcal(H)]$ is a subalgebra of $\lcal([H,H])$. Therefore, by Proposition~\ref{p:GradedLie},
\[
d(H):=\dim_{\f_p}(H/([H,H]H^p))\le \cdim_{\lcal(H)}[\lcal(H),\lcal(H)]  \le c_{10} c_7 \log x+c_{11} d.
\]
Thus $d_i(\prod_{\pfr\in T} P_{\pfr}^{(1)}):=\max\{d(H)|\h[\prod_{\pfr\in T} P_{\pfr}^{(1)}:H]=p^i\} \le c_{14} \log x+ c_{11} i.$ and so, by~\cite[Proposition 1.6.2]{LS}, 
\[
a_{p^i}(\prod_{\pfr\in T} P_{\pfr}^{(1)})=\#\{H|\h [\prod_{\pfr\in T} P_{\pfr}^{(1)}:H]=p^i\}\le p^{\sum^{i}_{j=0} d_j(\prod_{\pfr\in T} P_{\pfr}^{(1)})}\le p^{c_{14}i\log x+c_{11} i^2}.
\]
Hence we have
\[
s_x(\prod_{\pfr\in T} P_{\pfr}^{(1)})\le \sum_{p^i\le x}a_{p^i}(\prod_{\pfr\in T} P_{\pfr}^{(1)})\le  \sum_{p^i\le x}  p^{c_{12}i\log x+c_{11} i^2}\le \log x\cdot x^{(c_{14}+c_{11}) \log x},
\]
which finishes our proof modulo Theorem~\ref{t:GradedLie}.
\end{proof}
\subsection{Graded Lie algebras.}\label{ss:GradedLie}
In this section, we prove Theorem~\ref{t:GradedLie}, which completes our proof.
 \begin{lem}~\label{l:LargeCommutator2}
Let $\gfr$ be Lie algebra, $\gfr_0$ and $\gfr_1$ two subspaces of $\gfr$, and $D$ a natural number; then for any $U$ and $V$ subspaces of $\gfr_0^D$ and $\gfr_1^D$, respectively, we have
\[
\cdim_{[\gfr_0,\gfr_1]^D} [U,V] \le \dim(\gfr_1)\h \cdim_{\gfr_0^D} U+\dim[\gfr_0,\gfr_1]\h\cdim_{\gfr_1^D} V.
\]
\end{lem}
\begin{proof}
Let $d=\dim \gfr_1$. As in~\cite[Section 6.3]{LS}, let us consider $M_{D,d}(\gfr_1)$ the set of all $D$ by $d$ matrices with entries in $\gfr_1$. For any ${\bf x}\in M_{D,d}(\gfr_1)$, let $W_i({\bf x})$ be the linear span of the entries in the $i^{th}$ row of ${\bf x}$ and $\rho({\bf x})=\sum_{i=1}^{D}\dim W_i({\bf x})$. Let $V^{(d)}$ be the the subspace of $M_{D,d}(\gfr_1)$ consisting of matrices whose columns are in $V$. Choose ${\bf x}$ in $V^{(d)}$ such that $\rho({\bf x})=\max_{{\bf y} \in V^{(d)}} \rho({\bf y})$. By permuting the rows of ${\bf x}$ and changing $V$ and $U$ if necessary, without loss of generality, we can assume that $W_i({\bf x})=\gfr_1$ for $i\le t$ and $W_i({\bf x})\neq \mfr$ for $t<i$, where $0\le t\le D$. Hence $\sum_{j=1}^{d} [\gfr_0,x_{ij}]=[\gfr_0,\gfr_1]$ for any $i\le t$, and so
\begin{equation}\label{e:CommutatorFullComponents2}
[\gfr_0,\gfr_1]^{t}\oplus 0^{D-t}\subseteq \sum_{j=1}^{d}[\gfr_0^D,\xvec_j],
\end{equation}
where $\xvec_j=(x_{1j},\cdots,x_{Dj})$ is the $j^{th}$ column of ${\bf x}$. On the other hand, for any $j$, we have that
\[
\dim [U,\xvec_j]+\dim U\cap \ker[\bullet,\xvec_j]=\dim U. 
\]
Hence $\dim [U,\xvec_j]\ge \dim [\gfr_0^D,\xvec_j]-\cdim_{\gfr_0^D} U$. Thus
\begin{equation}\label{e:CommutatorDimLower2}
\dim [U,V]\ge \dim \sum_{j=1}^d [U,\xvec_j] \ge \dim \sum_{j=1}^d [\gfr_0^D,\xvec_j]-d\h \cdim_{\gfr_0^D}U.
\end{equation}
By (\ref{e:CommutatorFullComponents2}) and (\ref{e:CommutatorDimLower2}), we have that
\[
\dim[U,V]\ge t\h \dim[\gfr_0,\gfr_1]-d\h\cdim_{\gfr_0^D}U,
\] 
which means 
\begin{equation}\label{e:CommutatorCodim12}
\cdim_{[\gfr_0,\gfr_1]^D}[U,V]\le (D-t) \dim [\gfr_0,\gfr_1]+d\h\cdim{\gfr_0^D}U.
\end{equation}
By (\ref{e:CommutatorCodim12}), it is enough to show that 
\[
(D-t) \dim[\gfr_0,\gfr_1]\le \dim[\gfr_0,\gfr_1]\h \cdim_{\gfr_1^D}V.
\]
 Assume the contrary, i.e. $D-t> \cdim_{\gfr_1^d}V$. In particular, $D>t$, i.e. $W_D({\bf x})$ is a proper subspace of $\gfr_1$. Since $\dim \gfr_1=d$, without loss of generality, we can assume that 
\[
{\rm span}\langle x_{D1},\cdots,x_{Dd}\rangle={\rm span}\langle x_{D2},\cdots,x_{Dd}\rangle.
\]
 For any $i$ larger than $t$, take $y_i\in\gfr_1\setminus W_i({\bf x})$, and let $\hat{y}_i$ be a vector in $\gfr_1^D$ whose only non-zero entry is the $i^{th}$ entry, $y_i$. By the contrary assumption, $V$ intersects ${\rm span}\langle\hat{y}_{t+1},\cdots,\hat{y}_D\rangle$ non-trivially. Let $0\neq\hat{z}\in V\cap {\rm span} \langle\hat{y}_{t+1},\cdots,\hat{y}_D\rangle$. Without loss of generality, we can assume that $\hat{z}_D\neq 0$ and define ${\bf y}$ as follows:
 \[
 y_{ij}=\left\{\begin{array}{cc} x_{ij}&{\rm if}\h j\neq 1,\\ \hat{z}_i+x_{i1}&{\rm if}\h j= 1.\end{array}\right.
 \]
 Therefore ${\bf y}$ is in $V^{(d)}$. On the other hand, $W_i({\bf y})=W_i({\bf x})$ for $i\le t$, and it is easy to see that $\dim W_i({\bf y})\ge \dim W_i({\bf x})$ for $i> t$ where the equality does not hold at least for $i=D$. So $\rho({\bf y})>\rho({\bf x})$ which is a contradiction.
 \end{proof}
\begin{remark}
\begin{itemize}
\item[1-] What is important in this lemma is the fact that the coefficients $\dim \gfr_1$ and $\dim[\gfr_0,\gfr_1]$ are independent of $D$. 
\item[2-] When $\gfr$ is a perfect Lie algebra and $\gfr_0=\gfr_1=\gfr$, this lemma is proved in~\cite[Section 6.2]{LS}. Our argument is a modification of theirs. 
\item[3-] When $\gfr$ is a split Lie algebra and $\gfr_0=\gfr_1=\gfr$, in \cite{ANS}, Ab\'{e}rt, Nikolov, and Szegedy prove a similar inequality with coefficient 2 instead of $\dim(\gfr)$.
\end{itemize}
\end{remark}
\begin{proof}[Proof of Theorem~\ref{t:GradedLie}]
For any $x=\sum^{\infty}_{i=1} x_i t^i$ in $\mathfrak{L}^D$, let ${\rm deg}(x)$ be the smallest integer $i$ such that $x_i$ is not zero, and $ld(x)=x_{{\rm deg}(x)}$. For any $n$, let
\[
\hfr_n=\{ld(x)\h|\h x\in \hfr, {\rm deg}(x)=n\}\cup\{0\}.
\]
In a similar fashion, we can define $[\hfr,\hfr]_n$. Clearly $\hfr_n$ and $[\hfr,\hfr]_n$ are subspaces of $\gfr_n^D$, for any $n$. One can also see that 
\begin{equation}\label{e:codim}
\cdim_{\mathfrak{L}^D}\hfr=\sum^{\infty}_{n=1}\cdim_{\gfr_n^D}\hfr_n,
\end{equation}
and a similar formula holds for the codimension of $[\hfr,\hfr]$. On the other hand, since, for any $n\ge 2m$,
\[
\sum^{m-1}_{i=0} [\hfr_{a(n,i)},\hfr_{b(n,i)}]\subseteq [\hfr,\hfr]_n,
\]
where $a(n,i)=m\lfloor n/(2m) \rfloor+i$ and $b(n,i)=n-a(n,i)$, we have
\[
\cdim_{\gfr^D_n} [\hfr,\hfr]_n\le \cdim_{\gfr^D_n} \sum^{m-1}_{i=0} [\hfr_{a(n,i)},\hfr_{b(n,i)}].
\]
Hence we have
\begin{equation}\label{e:reduction}
\cdim_{\mathfrak{L}^D}[\hfr,\hfr] \le \sum^{\infty}_{n=1} \cdim_{\gfr^D_n} \sum^{m-1}_{i=0} [\hfr_{a(n,i)},\hfr_{b(n,i)}].
\end{equation}
By Lemma~\ref{l:LargeCommutator2}, we know that, for any $0\le i\le m-1$ and $2m\le n$,
\[
\cdim_{[\gfr^D_{a(n,i)},\gfr^D_{b(n,i)}]} [\hfr_{a(n,i)},\hfr_{b(n,i)}]\le \dim\widehat{\gfr}\h(\cdim_{\gfr_{a(n,i)}^D} \hfr_{a(n,i)}+\cdim_{\gfr_{b(n,i)}^D} \hfr_{b(n,i)}).
\]
Since $\widehat{\gfr}$ is a perfect Lie algebra and $\{a(n,i)\h|\h0\le i\le m-1\}$ is a complete residue system mod $m$,  
\[
\gfr_n=\sum^{m-1}_{i=0}[\gfr_{a(n,i)},\gfr_{b(n,i)}],
\]
for any $2m\le n$. Thus, for $2m\le n$, we have
\[
\cdim_{\gfr^D_n} \sum^{m-1}_{i=0} [\hfr_{a(n,i)},\hfr_{b(n,i)}] \le \sum^{m-1}_{i=0} \cdim_{[\gfr^D_{a(n,i)},\gfr^D_{b(n,i)}]} [\hfr_{a(n,i)},\hfr_{b(n,i)}]\hspace{2.25cm}
\]
\[
\hspace{4.25cm}\le\dim\widehat{ \gfr} \sum^{m-1}_{i=0} (\cdim_{\gfr_{a(n,i)}^D} \hfr_{a(n,i)}+\cdim_{\gfr_{b(n,i)}^D} \hfr_{b(n,i)}).
\]
Combining it with the inequality \ref{e:reduction}, we get
\[
\cdim_{\mathfrak{L}^D} [\hfr,\hfr]\le 2\dim\widehat{\gfr}\cdot D+ \dim\widehat{\gfr} \sum^{\infty}_{n=1} \sum^{m-1}_{i=0}  (\cdim_{\gfr_{a(n,i)}^D} \hfr_{a(n,i)}+\cdim_{\gfr_{b(n,i)}^D} \hfr_{b(n,i)}).
\]
It is also easy to see that any positive integer $k$ is equal to either $a(n,i)$ or $b(n,i)$ for at most $4m$ paris of numbers $(n,i)$. Therefore overall we have
\[
\cdim_{\mathfrak{L}^D} [\hfr,\hfr]\le 2\dim\widehat{\gfr}\cdot D+ 4m \dim\widehat{\gfr} \sum^{\infty}_{k=1}\cdim_{\gfr_{k}^D} \hfr_{k}\le C (D+ \cdim_{\mathfrak{L}^D}\hfr),
\]
where $C= 4m \dim\widehat{\gfr}$, which finishes the proof.
\end{proof}

\subsection{A lower bound on the number of subgroups.}~\label{ss:LatticesLowerBound}
In order to get the lower bound of Theorem~\ref{t:CountingLattices}, we essentially follow Shalev's idea~\cite{Sh}. However, we have to be extra careful as we are counting up to  an automorphism of $G$.
\\

\noindent
In fact, we show that inside any maximal lattice we can find at least the claimed number of subgroups which are distinct even up to an automorphism of $G$. 

\begin{lem}\label{l:LowerboundSubgroupMax}
Let $\Gamma$ be a maximal lattice in $G$. Then 
\begin{enumerate}
\item For any $x>x_0$ 
\[
x^{c\log x/\log(\vol(G/\Gamma))}\le s_x(\Gamma)
\]
 where $c$ and $x_0$ just depend on $G$. 
\item For any $x>x_0$
\[
x^{c\log x}\le |\{\Lambda \subseteq\Gamma|\h \text{$\Lambda$ is a subgroup of $\Gamma$},\h [\Gamma:\Lambda]\le x\}/\sim|,
\]
where $x_0$ just depends on $G$, $c$ depends on $\Gamma$ and $\Lambda_1\sim \Lambda_2$ if and only if there is an automorphism $\theta$ of $G$ such that $\Lambda_2=\theta(\Lambda_1)$.
\end{enumerate}
\end{lem}
\begin{proof}
Since $\Gamma$ is a maximal lattice, we can find $k$, $\pfr_0$, $\bbg$, $\Lambda$ and
 $\{P_{\pfr}\}_{\pfr\in V_k^{\circ}}$ as described in Section~\ref{s:maximal}. By the discussion in Section~\ref{ss:Type}, we know that if $\deg(\pfr)\gg \log v$ where $v=\vol(G/\Gamma)$ and the implied constant just depends on $G$, then $P_{\pfr}$ is hyper-special. On the other hand, by \cite{Sc}, we
 know that the group of automorphisms of $k$ is finite. Hence, by Chebotarev's density
 Theorem~\cite{FJ}, the set of places $\pfr$ of $k$ whose decomposition group, i.e. $\{\theta\in \Aut(k)|\h \theta(\pfr)=\pfr\}$, is trivial has positive density. Now in order to prove the first part, 
let $\pfr_1$ be any place with degree $O(\log v)$ (the implied constants depend only on $G$) such 
that $P_{\pfr_1}$ is hyper-special and for the second part, let $\pfr_1$ be a place such that 
$P_{\pfr_1}$ is hyper-special and also the decomposition group of $\pfr_1$ is trivial.  For any 
positive integer $n$, $P_{\pfr_1}^{(n)}/P_{\pfr_1}^{(2n)}$ is an $\f_{\pfr_1}$-vector space of 
dimension $n\dim \bbg$. So $P_{\pfr_1}^{(n)}/P_{\pfr_1}^{(2n)}$ has at least $q_{\pfr_1}^{\lfloor (n\dim\bbg)/2\rfloor^2}$ subspaces. Preimage of any such subspace $W$ gives us $Q_W$ a 
subgroup of $P_{\pfr_1}^{(n)}$ which contains $P_{\pfr_1}^{(2n)}$. Let 
\[
\Lambda_W=\bbg(k)\cap \prod_{\pfr\in V_k^{\circ}\setminus\{\pfr_1\}} P_{\pfr}\cdot Q_W.
\]
By strong approximation, we have that $[\Lambda:\Lambda_W]=q_{\pfr_1}^{n \dim\bbg-\dim_{\f{\pfr_1}}W}\le q_{\pfr_1}^{n \dim \bbg}$. 
\\

\noindent
This shows that 
\[
s_{q_{\pfr_1}^{n\dim \bbg}}(\Lambda)\ge q_{\pfr_1}^{\lfloor(n\dim \bbg)/2\rfloor^2},
\]
which (coupled with the way we chose $\pfr_1$) implies that $s_x(\Lambda)\ge x^{c'_2\log x/\log v}$, where $v=\vol(G/\Gamma)$ and $c'_2$ only depends on $G$. Therefore we have 
\begin{equation}\label{e:Ineq}
\log s_x(\Gamma)\gg (\log(x/|\Gamma/\Lambda|))^2/\log v,
\end{equation}
where the implied constant only depends on $G$.
For $x\ge |\Gamma/\Lambda|^{3/2}$, we have $\log (x/|\Gamma/\Lambda|)\ge \frac{1}{3}\log x$ and so by (\ref{e:Ineq}) $\log s_x(\Gamma)\gg (\log x)^2/\log v$ as we wished. On the other 
hand, by the discussion in Section~\ref{s:maximal}, $\Gamma/\Lambda$ is a finite abelian group  
and its quotient by a subgroup of order at most $t(\bbg)^{\vare(\bbg)}$ is $t(\bbg)$-torsion. It is 
worth mentioning that $t(\bbg)$ and $\vare(\bbg)$ have upperbounds which just depend on $G$.  
Hence for some prime factor $r$ of $t(\bbg)$ we have that 
\[
d_r(\Gamma/\Lambda):=\dim_{\f_r}((\Gamma/\Lambda)/(\Gamma/\Lambda)^r)\gg \log |\Gamma/\Lambda|,
\]
where the implied constant only depends on $G$. Thus, for $x\gg 0$, we have
\begin{equation}\label{e:Ineq2}
\log s_x(\Gamma)\ge \log s_x(\Gamma/\Lambda)\gg \log |\Gamma/\Lambda|,
\end{equation}
where the implied constants only depend on $G$. On the other hand, by Lemma \ref{l:gamma/lambda}, we have $\log |\Gamma/\Lambda|\ll \log v$. Thus by (\ref{e:Ineq2}) we have
\[
\log s_x(\Gamma)\gg (\log |\Gamma/\Lambda|)^2/\log v.
\]
Hence for $x< |\Gamma/\Lambda|^{3/2}$ we have $\log s_x(\Gamma)\gg  (\log x)^2/\log v,$ as we wished.
\\

\noindent
 To get the second part, we have to show that only small number of these subgroups are equivalent, which is proved in the next lemma. 

\begin{lem}\label{l:OrbitAut}
For a given $W$ as above there are at most $q_{\pfr_1}^{cn}$ subspaces $W'$ such that $\theta(\Lambda_W)=\Lambda_{W'}$ for some $\theta\in \Aut(G)$, where $c$ only depends on $G$.
\end{lem}
\begin{proof}
Assume that $\theta(\Lambda_W)=\Lambda_{W'}$; then, by a similar argument as in the proof of Lemma~\ref{l:LatticeDiffer}, there are $\theta_1\in \Aut(k)$ and $\theta_2:\leftexp{\theta_1}{\bbg}\rightarrow \bbg$ a $k$-isomorphism such that $\theta=\theta_2\circ\theta_1$. Hence for any $\pfr\in V_k^{\circ}$ the closure of $\Lambda_W$ in $\bbg(k_{\pfr})$ is isomorphic to the closure of $\Lambda_{W'}$ in $\bbg(k_{\theta_1(\pfr)})$. On the other hand, by strong approximation, the closure of $\Lambda_W$ is either $P_{\pfr}$ if $\pfr\neq \pfr_1$ or $Q_W$ if $\pfr=\pfr_1$ and $P_{\pfr}$'s are parahoric subgroups and $Q_W$ is not. Thus $\theta_1(\pfr_1)=\pfr_1$. So by the way that we chose $\pfr_1$, we have that $\theta_1$ is trivial, $\theta\in (\Aut\h\bbg)(k)$, and moreover $\theta\in \prod_{\pfr\in V_k^{\circ}\setminus \{\pfr_1\}} \widetilde{P}_{\pfr}$, where 
\[
\widetilde{P}_{\pfr}=\{\theta_{\pfr}\in (\Aut\h\bbg)(k_{\pfr})|\h \theta_{\pfr}(P_{\pfr})=P_{\pfr}\}.
\]
It also belongs to 
\[
\widetilde{N}(Q_W,P_{\pfr}^{(n)})=\{\theta_{\pfr_1}\in (\Aut\h\bbg)(k_{\pfr_1})|\h P_{\pfr_1}^{(2n)}\subseteq \theta_{\pfr_1}(Q_W)\subseteq P_{\pfr_1}^{(n)}\}.
\]
Overall the number that we are interested in is at most the number of left cosets of 
\[
(\Aut\h\bbg)(k)\cap \prod_{\pfr\in V_k^{\circ}\setminus\{\pfr_1\}}\widetilde{P}_{\pfr}\cdot \widetilde{Q}_W
\]
in
\[
(\Aut\h\bbg)(k)\cap \prod_{\pfr\in V_k^{\circ}\setminus\{\pfr_1\}}\widetilde{P}_{\pfr}\cdot \widetilde{N}(Q_W,P_{\pfr}^{(n)}),
\]
where $\widetilde{Q}_W=\{\theta_{\pfr_1}\in (\Aut\h\bbg)(k_{\pfr_1})|\h \theta_{\pfr_1}(Q_W)=Q_W\}.$
So it is at most the number of left cosets of $\widetilde{Q}_W$ in $\widetilde{N}(Q_W,P_{\pfr_1}^{(n)})$,
which is at most
\[
\#(\Aut\h\bbg)(k_{\pfr_1})/\Ad(\bbg(k_{\pfr_1}))\cdot\# N(Q_W,P_{\pfr_1}^{(n)})/N_{\bbg(k_{\pfr_1})}(Q_W),
\]
where $N(Q_W,P_{\pfr}^{(n)})=\{g\in\bbg(k_{\pfr_1})|\h P_{\pfr_1}^{(2n)}\subseteq g Q_W g^{-1}\subseteq P_{\pfr_1}^{(n)}\}$. As the first factor just depends on $G$, without loss of generality we will give an upper bound for the second factor. $P_{\pfr_1}$ acts from left on $X=N(Q_W,P_{\pfr_1}^{(n)})/N_{\bbg(k_{\pfr_1})}(Q_W)$ and
\begin{equation}\label{e:NumberLeftCoset}
\#X=\sum_{[g]\in P_{\pfr_1}\!\!\setminus X} \#P_{\pfr_1}/(P_{\pfr_1}\cap N_{\bbg(k_{\pfr_1})}(gQ_Wg^{-1})).
\end{equation}
Since $g\in N(Q_W,P_{\pfr_1}^{(n)})$, we have that 
$
P_{\pfr_1}^{(2n)}\subseteq g Q_W g^{-1}\subseteq P_{\pfr_1}^{(n)}.
$
Hence $P_{\pfr_1}^{(n)}$ is contained in $N_{\bbg(k_{\pfr_1})}(gQ_Wg^{-1})$ and so, by equation~(\ref{e:NumberLeftCoset}), the number of elements of $X$ is at most $\#P_{\pfr_1}\!\!\!\setminus\! X\cdot \#P_{\pfr_1}/P_{\pfr_1}^{(n)}$. Because the latter term is at most $q_{\pfr_1}^{c_{15}n}$, where $c_{15}$ just depends on $G$, it is enough to give the right upper bound for the number of elements of 
\[
P_{\pfr_1}\!\!\!\setminus\! N(Q_W,P_{\pfr_1}^{(n)})/N_{\bbg(k_{\pfr_1})}(Q_W).
\]
It is clear that the number of discussed double cosets is at most equal to the number of right cosets of $P_{\pfr_1}$ in $Y=\{g\in \bbg(k_{\pfr_1})|\h g P_{\pfr_1}^{(2n)}g^{-1}\subseteq P_{\pfr_1}^{(n)}\}$. By a similar argument as above, we have
\[
\#P_{\pfr_1}\!\!\!\setminus \!Y=\sum_{[g]\in P_{\pfr_1}\!\!\setminus Y/P_{\pfr_1}} \# (g^{-1}P_{\pfr_1}g\cap P_{\pfr_1})\!\!\setminus P_{\pfr_1}\le \#P_{\pfr_1}\!\!\setminus \!Y/P_{\pfr_1}\cdot \#P_{\pfr_1}^{(2n)}\!\!\setminus P_{\pfr_1}.
\]
Again the latter term is at most $q_{\pfr_1}^{c_{16}n}$, where $c_{16}$ only depends on $G$, so it is enough to give the right upper bound for the number of double cosets of $P_{\pfr_1}$ in $Y$. The latter is a direct consequence of the definitions of $Y$ and $P_{\pfr_1}^{(n)}$'s, and the Cartan decomposition~\cite[Section 3.3.3]{T}.
\end{proof}
To complete the proof of Lemma~\ref{l:LowerboundSubgroupMax}, it is enough to notice that, for any positive integer $n$, $\Lambda$ has at least $p^{c_{17}n^2}$ many subgroups $\Lambda_W$'s of index at most $p^{c_{18} n}$,  where $c_{17}, c_{18} $ only depend on $\Gamma$. On the other hand, by Lemma~\ref{l:OrbitAut}, each orbit of $\Aut(G)$ intersects this set of lattices in a set of at most $p^{c_{19} n}$ many elements, where $c_{19}$ only depends on $G$. Hence, overall, the proof of the lower bound is completed as we get  $p^{c_{20}n^2}$ many lattices of index at most $p^{c_{18} n}$ in $\Lambda$, which are distinct up to automorphisms of $G$.
\end{proof}
\begin{proof}[Proof of the lower bound of Theorem~\ref{t:CountingLattices}]
We fix a maximal lattice in $G$ and apply the second part of Lemma~\ref{l:LowerboundSubgroupMax}.
\end{proof}
\subsection{Completion of the proofs.}~\label{ss:SubgroupGrowth}
\begin{proof}[Proof of Theorem~ \ref{t:SubgroupGrowth}] By the definition of $\Lambda_0$, it is 
equal to $\bbg(k)\cap\prod_{\pfr\in V_k^{\circ}} Q_{\pfr}$, where $Q_{\pfr}$ is an open compact 
subgroup of $\bbg(k_{\pfr})$. So, for any $\pfr$, there is a maximal parahoric subgroups 
$P_{\pfr}$ which contains $Q_{\pfr}$. Note that for almost all places, $P_{\pfr}=Q_{\pfr}$ and it 
is a hyper-special parahoric subgroup~\cite{T}. So the lower bound can be proved using a similar 
construction as in the proof of Lemma \ref{l:LowerboundSubgroupMax}. In fact, the argument 
here is much easier as the implied constants can depend on $\Lambda_0$. We take any $\pfr$ such that $Q_{\pfr}=P_{\pfr}$ is a hyper-special parahoric and construct subgroups of $\Lambda_0$ using subspaces of $P^{(n)}_{\pfr}/P^{(2n)}_{\pfr}$ as in the proof of Lemma \ref{l:LowerboundSubgroupMax}.
\\

\noindent
For the upper bound, we notice that $\Lambda_0$ is 
contained in $\Lambda=\bbg(k)\cap\prod_{\pfr\in V_k^{\circ}} P_{\pfr}$  and $c_x(\Lambda_0)\le c_{[\Lambda:\Lambda_0]x}(\Lambda)$, and so it is enough to give the right upper bound for the 
number of congruence subgroups of $\Lambda$. By the strong approximation, we have to estimate 
$s_x(\prod_{\pfr\in V_k^{\circ}}P_{\pfr})$, which has been done in the second half of Section \ref{ss:ReductionProP} and Section \ref{ss:ReductionLieAlgebra}.
\end{proof}
\begin{proof}[Proof of Corollary~ \ref{c:SubgroupGrowth}] 
Here we can and will normalize the Haar measure in a way that the covolume of any lattice is at least 1 (by changing $C$, $D$ and $x_0$ if necessary.). We first prove the upper bound. By Theorem~\ref{t:MaxSubgroupGrowthUpperBound}, we have
\[
\log s_y(\Gamma_{max})\ll (\log y)^2,
\]
for any maximal lattice $\Gamma_{max}$ and $y\gg \vol(G/\Gamma)$, where the implied constants only depend on $G$. In particular, for any $x\ge x_0$ (where $x_0$ only depends on $G$) and any maximal lattice $\Gamma_{max}$ in $G$, we have
\[
\log s_x(\Gamma_{max})\le \log s_{v_m x}(\Gamma_{max})\ll (\log (v_m x))^2,
\]
where the implied constant only depend on $G$ and $v_m=\vol(G/\Gamma_{max})$.
  For an arbitrary lattice $\Gamma$ in $G$, let $\Gamma_{max}$ be a maximal lattice in $G$ containing $\Gamma$. Then
\[
\log s_x(\Gamma)\le \log s_{|\Gamma_{max}/\Gamma|x}(\Gamma_{max})\ll (\log (|\Gamma_{max}/\Gamma| v_m x))^2=(\log(\vol(G/\Gamma) x))^2.
\]
This gives us the claimed upper bound.
\\

\noindent
For a given lattice $\Gamma$ in $G$, let $\Gamma_{max}$ be a maximal lattice containing it. 
Furthermore let $\Lambda_{max}$, $\bbg$ and $\{P_{\pfr}\}$ be the parameters given by Rohlfs's criterion and 
$\Lambda=\Lambda_{max}\cap \Gamma$. Let $v_m=\vol(G/\Gamma_{max})$ and 
$v=\vol(G/\Gamma)=v_m [\Gamma_{max}:\Gamma]$. We notice that $\Gamma/\Lambda$ can be embedded into $\Gamma_{max}/\Lambda_{max}$. Hence 
\begin{enumerate}
\item By Lemma~\ref{l:gamma/lambda}, $\log |\Gamma/\Lambda| \ll \log v_m$, where the implied constant just depends on $G$.
\item $\Gamma/\Lambda$ is a finite abelian group. Its quotient by a subgroup of order 
$t(\bbg)^{\vare(\bbg)}$ is $t(\bbg)$-torsion. (Let us again recall that $t(\bbg)$ and $\vare(\bbg)$ 
have upper bounds which only depend on $G$.)
\end{enumerate}
 On the other hand, $|\Lambda_{max}/\Lambda|\le |\Gamma_{max}/\Gamma|\le v$. By  CSP and MP, the profinite closure $\widehat{\Lambda}$ of $\Lambda$ can be viewed as an open subgroup of the profinite closure $\widehat{\Lambda}_{\max}$ of $\Lambda_{\max}$. Therefore $\widehat{\Lambda}$ is a subgroup of index at most $v$ in  $\prod_{\pfr\in V_k^{\circ}} P_{\pfr}$. By the discussions in Section~\ref{ss:ReductionProP} and 
Lemma~\ref{l:Subnormal}, for any $\pfr$ with $\deg \pfr\gg \log v$ (where the implied constant only depends on $G$), we have that 
\begin{enumerate}
\item $P_{\pfr}$ is hyper-special.
\item $P_{\pfr}\subseteq \widehat{\Lambda}$.
\end{enumerate}
Thus by an identical argument as in the proof of the first part of Lemma~\ref{l:LowerboundSubgroupMax}, one can finish the proof. 
\end{proof}

\subsection{Appendix: table of notations.}
\[
\begin{array}{l|l}
p& \mbox{A prime number larger than 3}\\
K& \mbox{A local field of characteristic $p$}\\
\bbg_0& \mbox{A simply connected absolutely almost simple $K$-group}\\
G & \bbg_0(K)\\
k&  \mbox{A function field}\\
q_k& \mbox{Number of elements of the constant field of $k$}\\
g_k& \mbox{The genus of $k$}\\
h_k& \mbox{The class number of $k$}\\
Br(\bullet)&\mbox{The Brauer group}\\
l& \mbox{A finite extension of $k$ of degree either $2$ or $3$}\\
h_l& \mbox{The class number of $l$}\\
V_k& \mbox{The set of equivalence classes of all the places on $k$}\\
\Div(k)&\mbox{The group of divisors of $k$}\\
\Div^+(k)&\mbox{The set of effective divisors of $k$}\\
\Div(T)&=\sum_{\pfr\in T} \pfr\\
\deg(T)&=\sum_{\pfr\in T}\deg \pfr\\
\mathcal{V}_k(y)&=\{T\subseteq V_k^{\circ}| \sum_{\pfr\in T}\deg(\pfr)\le y\}\\
l(D)&=\dim_{\f_k}(\{x\in k^{\times}| (x)+D\ge 0\}\cup\{0\}) \h\mbox{for a given divisor $D$}\\
k_{\pfr}& \mbox{The $\pfr$-adic completion of $k$}\\
\o_{\pfr} &\mbox{The ring of integers in $k_{\pfr}$}\\
\f_{\pfr}& \mbox{The residue field of $k_{\pfr}$}\\
\o(\pfr)&\mbox{The ring of $\pfr$-integers in $k$}\\
\bba_k&\mbox{The ring of adeles of $k$}\\
\pfr_0 & \mbox{A fixed element of $V_k$ such that $k_{\pfr}\simeq K$}\\
V^{\circ}_k&=V_k\setminus\{\pfr_0\}\\
\mu_n&\mbox{The schematic kernel of $x\mapsto x^n$ from $\bbg_m$ to itself}\\
R_{l/k}^{(1)}(\mu_n)&\mbox{The schematic kernel of the norm map from $R_{l/k}(\mu_n)$ to $\mu_n$}\\
\bbg & \mbox{A simply connected absolutely almost simple $k$-group}\\
\mu & \mbox{The center of $\bbg$}\\
\adg & \mbox{The adjoint form of $\bbg$}
\end{array}
\]
\[
\begin{array}{l|l}
\gcal& \mbox{The unique quasi-split inner $k$-form of $\bbg$}\\
\adgcal& \mbox{The adjoint form of $\gcal$}\\
\bbg & \simeq \bbg_0\h \mbox{over}\h K\\
P_{\pfr}& \mbox{A parahoric subgroup in  $\bbg(k_{\pfr})$}\\
P_{\pfr}^m& \mbox{A ``large" parahoric subgroup whose intersection with $P_{\pfr}$ contains an Iwahori}\\
\pcal_{\pfr}& \mbox{One of the ``largest" parahoric subgroups of $\gcal(k_{\pfr})$}\\
\Theta_{\pfr}&\mbox{The type of the parahoric subgroup $P_{\pfr}$}\\
\overline{P}_{\pfr}&\mbox{The set of stabilizer of $P_{\pfr}$ in $\adbbg(k_{\pfr})$}\\
\Cl(\adbbg,\{\overline{P}_{\pfr}\}_{\pfr\in V_k^{\circ}})&=\adbbg(\bba_k)/(\adbbg(k)\cdot\adbbg(k_{\pfr_0})\prod_{\pfr\in V_k^{\circ}}\overline{P}_{\pfr})\h \mbox{(The class group of $\adbbg$ w.r.t. $\{\overline{P}_{\pfr}\}$)}\\
\cl(\adbbg,\{\overline{P}_{\pfr}\}_{\pfr\in V_k^{\circ}})&=\#\Cl(\adbbg,\{\overline{P}_{\pfr}\}_{\pfr\in V_k^{\circ}})\h \mbox{(The class number of $\adbbg$ w.r.t. $\{\overline{P}_{\pfr}\}$)}\\
P_{\pfr}^{(i)}&\mbox{The $i^{th}$ congruence subgroup of $P_{\pfr}$ (except in Lemma~\ref{l:Type})}\\
e(\pfr)& \mbox{Local factor in covolume formula}\\
e_m(\pfr)& \mbox{Local factor associated with $P_{\pfr}^m$}\\
e_{qs}(\pfr)& \mbox{Local factor associated with $\pcal_{\pfr}$}\\
e'(\pfr) & \mbox{Correctional factor in the covolume formula}\\
Z(\gcal) & \mbox{The product of $e_{qs}(\pfr)$}\\
\sfr(\gcal) & \mbox{A number which depends on the type of $\gcal$ (ref. \ref{s:Covolume})}\\
B(\gcal) &=q_k^{(g_k-1)\dim \gcal} (q_l^{g_l-1}/q_k^{(g_k-1)[l:k]})^{\sfr/2}\\
\vol & \mbox{The Haar measure on $G$ such that $\vol(P_{\pfr_0})=1$}\\
\vare(\bbg)&=2 ({\rm resp.} =1) \mbox{if $\bbg$ is of type ${\rm D}_r$ with $r$ even (resp. otherwise)}\\
t(\bbg)& \mbox{The exponent of $\mu$} \\
T(\bbg)& =\{\pfr\in V_k|\h\text{ $\bbg$ splits $/ \widehat{k}_{\pfr}$ and $\bbg$ is not quasi-split $/ k_{\pfr}$}\}\\
T_c&\mbox{The set of ramified primes  at the level of commensurability (ref. Def.\ref{d:Ramified})}\\
T_l&\mbox{The set of ramified primes at the local level (ref. Def.\ref{d:Ramified})}\\
\mathfrak{R}(\Gamma)&=T_c\cup T_l\\
\dcal_{\pfr} & \mbox{The local Dynkin diagram of $\bbg/k_{\pfr}$}\\
\Xi_{\pfr}& \mbox{The subgroup of $\Aut\h \dcal_{\pfr}$ coming from $\adbbg(k_{\pfr})$}\\
\xi_{\pfr} & \mbox{The homomorphism from $H^1(k,\mu)$ to $\Aut\h \dcal_{\pfr}$}\\
\xi^{\circ} & =(\xi_{\pfr})_{\pfr\in V_k^{\circ}}\\
\xi & =(\xi_{\pfr})_{\pfr\in V_k}\\
H^1(k,\mu)_{\xi^{\circ}} & =\ker \xi^{\circ}\\
H^1(k,\mu)_{\xi} & =\ker \xi\\
\bullet^{\times}&\mbox{The group of units of the given ring}\\
\mathfrak{M}_x&\mbox{A set of representatives of maximal lattices in $G$ with covolume at most $x$}\\
\rho_x(G)&\mbox{Number of lattices with covolume at most $x$ in $G$, up to an automorphism of $G$}\\
m_x(G)&\mbox{The number of maximal lattices, up to $\Aut(G)$, with covolume at most $x$.}\\
s_x({\bullet})&\mbox{Number of subgroups of index at most $x$}\\
s_x^{\triangleleft\triangleleft}({\bullet})&\mbox{Number of subnormal subgroups of index at most $x$}\\
c_x(\bullet)&\mbox{Number of congruence subgroups of index at most $x$}\\
a_x(\bullet)&\mbox{Number of subgroups of index $x$}\\
d(\bullet)&\mbox{Minimum number of elements of a generating set of a pro-$p$ group}\\
d_i(\bullet)&\mbox{Maximum of $d(\bullet)$'s for all the subgroups of index $p^i$}
\end{array}
\]
{\sc Dept. of Math, Univ. of California, San Diego, CA, 92093-0112}\\
{\it E-mail address:}{\tt: asalehigolsefidy@ucsd.edu}

\begin{thebibliography}{zzzzz}
\bibitem[ANS03]{ANS}M.~Ab\'{e}rt, N.~Nikolov, B.~Szegedy,
{\it Congruence subgroup growth of arithmetic groups in positive characteristic},
Duke Math. J. {\bf 117} no. 2 (2003) 367-383.

\bibitem[BCP82]{BCP}L.~Babai, P.~J.~Cameron, P.~P.~Palfy,
{\it On the orders of primitive groups with restricted non-abelian composition factors,}
J. algebra {\bf 79} (1982) 161-168.

\bibitem[Be07]{Be}M.~Belolipetsky,
{\it  Counting maximal arithmetic subgroups,}
Duke Math. J. {\bf 140} (2007), no. 1, 1-33, with an appendix by J. Ellenberg and
A. Venkatesh.

\bibitem[BGLS]{BGLS} M.~Belolipetsky, T. Gelander, A.~Lubotzky, A. Shalev,
{\it Counting arithmetic lattices and surfaces,}
Preprint.

\bibitem[BL]{BL} M.~Belolipetsky, A.~Lubotzky,
{\it Counting manifolds and class field towers,}
Preprint.

\bibitem[BGLM02]{BGLM} M. Burger, T. Gelander, A. Lubotzky, S. Mozes, 
{\it Counting hyperbolic manifolds,} 
Geom. and funct. anal. {\bf 12} (2002), no. 6, 1161-1173.

\bibitem[BP89]{BP}A.~Borel, G.~Prasad, 
{\it Finiteness theorems for discrete subgroups of bounded co-volume in semisimple groups,}
Publications math\'{e}matiques de l'I.H.E.S.  {\bf 69} (1989) 119-171.

\bibitem[BP90]{BPAdd}A.~Borel, G.~Prasad, 
{\it Addendum: Finiteness theorems for discrete subgroups of bounded covolume in semisimple groups,} 
Publications math\'{e}matiques de l'I.H.E.S. {\bf 71} (1990) 173-177.

\bibitem[BT87]{BrTi} F.~Bruhat, J.~Tits,
{\it Groupes r\'{e}ductifs sur un corps local I\!I\!I,}
J. fac. sci. Univ. Tokyo {\bf 84} (1987) 671-698.

\bibitem[dJK]{dJK} A.~J.~de~Jong, N.~Katz,
{\it Counting the number of curves over a finite field,}
Preprint, http://www.math.columbia.edu/$\sim$dejong/papers/curves.dvi

\bibitem[FJ05]{FJ} M.~D.~Fried, M.~Jarden,
 Field arithmetic,
Springer-Verlag,  Berlin, 2005.

\bibitem[G71]{G71} F.~Giraud,
 Cohomologie non ab\`{e}liennne,
 Grund. Math. Wiss., Springer-Verlag, 1971.
 
\bibitem[GLP04]{GLP} D. Goldfeld, A. Lubotzky, L. Pyber,
{\it Counting congruence subgroups,}
Acta math. {\bf 193} (2004) 73-104.
 
\bibitem[H74]{Ha74} G.~Harder,
{\it Chevalley groups over function fields and automorphic forms,}
Ann. Math. {\bf 100} (1974) 249-306.

\bibitem[H75]{Ha} G.~Harder,
{\it \"{U}ber die Galoiskohomologie halbeinfacher algebraischer Gruppen I\!I\!I,}
J. Reine angew. math. {\bf 274/275} (1975) 125-138.

\bibitem[Lu95]{Lu} A.~Lubotzky,
{\it Subgroup growth and congruence subgroups,}
Inventiones mathematicae {\bf 119}  (1995) 267-295.

\bibitem[LN04]{LN} A.~Lubotzky, N.~Nikolov,
{\it Subgroup growth of lattices in semisimple Lie groups,}
Acta math. {\bf 193} (2004) 105-139.

\bibitem[LS94]{LSa} A.~Lubotzky, A.~Shalev,
{\it On some $\Lambda$-analytic pro-$p$ group},
Israel J. Math. {\bf 85} (1994) 307-337.

\bibitem[LS03]{LS}A.~Lubotzky, D.~Segal,
Subgroup growth,
Birkh\"{a}user Verlag, 2003.

\bibitem[M77]{Mstrong}G.~Margulis, 
{\it Cobounded subgroups in algebraic groups over local fields,} 
Functional Anal. Appl. {\bf 11} (1977) 119-122.

\bibitem[M91]{Mar}G.~A.~Margulis,
Discrete subgroups of semisimple Lie groups,
Springer-Verlag, Berlin, 1991.

\bibitem[M80]{Mi}J.~S.~Milne,
\'{E}tale cohomology,
Princeton University Press, 1980.

\bibitem[MP94]{MP} A.~Moy, G.~Prasad,
{\it  Unrefined minimal $K$-types for $p$-adic groups,}
Invent. Math. {\bf 116} no. 1-3 (1994) 393-408. 

\bibitem[Pr75]{Pra} G.~Prasad,
{\it Triviality of certain automorphisms of semisimple groups over local fields,}
Mathematische Annalen {\bf 218} (1975) 219-227.

\bibitem[Pr77]{Pstrong} G.~Prasad,
{\it Strong approximation,}
Annals of math., 2nd series, {\bf 105} (1977) 553-572.

\bibitem[Pr89]{P} G.~Prasad,
{\it Volumes of $S$-arithmetic quotients of semisimple groups,}
Publications mat\'{e}matiques de l'I.H.E.S. {\bf 69} (1989) 91-117. 

\bibitem[PR84]{PR}G.~Prasad, M.~Raghunathan,
{\it Topological central extensions of semisimple groups over local fields,}
Annals of math., 2nd series, {\bf 119} no 1 (1984) 143-201.

\bibitem[PR]{PRsur} G.~Prasad, A.~Rapinchuk,
{\it Development on the congruence subgroup problem after the work of Bass, Milnor and Serre,}
Preprint. http://www.math.lsa.umich.edu/$\sim$gprasad/milnor090608.pdf

\bibitem[Ra76]{Rag1} M.~S.~Raghunathan,
{\it On the congruence subgroup problem,}
Publ. Math. IHES {\bf 46} (1976) 107-161.

\bibitem[Ra86]{Rag2} M.~S.~Raghunathan, 
{\it On the congruence subgroup problem II,}
Invent. Math. {\bf 85} (1986) 73-117.

\bibitem[R02]{Ro}M.~Rosen,
Number theory in function fields,
Springer-Verlag, New York, 2002. 

\bibitem[Sc38]{Sc}H.~L.~ Schmid,
{\it \"{U}ber die Automorphismen eines algebraischen Funktionenk\"{o}rpers von Primzahlcharakteristik,}
Jour. Reine Angew. Math. {\bf 179} (1938) 5-15.

\bibitem[T79]{T} J.~Tits,
{\it Reductive groups over local fields },
Proceedings of symposia in pure mathematics {\bf 33} (1979) 29-69.

\bibitem[Sc85]{Sch} W. Scharlau, 
Quadratic and hermitian forms, Springer-Verlag, Berlin, 1985.

\bibitem[Sh92]{Sh} A.~Shalev,
{\it Growth functions, $p$-adic analytic groups and groups of finite co-class,}
J.~London Math. Soc. {\bf 46} (1992) 111-122.

\bibitem[Sh72]{Shat} S.~S.~Shatz,
Profinite groups, arithmetic and geometry, 
Ann. of Math studies {\bf 67} Princeton University Press, 1972.

\bibitem[We82]{We} A.~Weil,
Adeles and algebraic groups,
Progress in Math. {\bf 23}, Birkh\"{a}user, Boston, 1982. 

\end{thebibliography}
\end{document}